\documentclass{amsart}
\usepackage{amsmath,amsthm}
\usepackage{amsfonts,amssymb}
\usepackage{amscd}
\usepackage{mathrsfs,url}

\numberwithin{equation}{section}

\newtheorem{theorem}{Theorem}[section]

\newtheorem{corollary}[theorem]{Corollary}
\newtheorem{proposition}[theorem]{Proposition}
\newtheorem{prop}[theorem]{Proposition}
\newtheorem{lemma}[theorem]{Lemma}

\theoremstyle{definition}
\newtheorem{definition}[theorem]{Definition}

\newtheorem{example}[theorem]{Example}
\newtheorem{problem}[theorem]{Problem}

\newcommand{\be}{\begin{equation}}
\newcommand{\ee}{\end{equation}}

\newcommand{\bdes}{\begin{description}}
\newcommand{\edes}{\end{description}}

\newcommand{\bal}{\begin{align}}
\newcommand{\eal}{\end{align}}

\newcommand{\bnum}{\begin{enumerate}}
\newcommand{\enum}{\end{enumerate}}

\newcommand{\bit}{\begin{itemize}}
\newcommand{\eit}{\end{itemize}}

\newcommand{\bea}{\begin{eqnarray}}
\newcommand{\eea}{\end{eqnarray}}

\newcommand{\bsry}{\begin{subarray}}
\newcommand{\esry}{\end{subarray}}

\newcommand{\bca}{\begin{cases}}
\newcommand{\eca}{\end{cases}}

\newcommand{\bcen}{\begin{center}}
\newcommand{\ecen}{\end{center}}

\newcommand{\bbm}{\begin{bmatrix}}
\newcommand{\ebm}{\end{bmatrix}}

\newcommand{\bmx}{\begin{matrix}}
\newcommand{\emx}{\end{matrix}}

\newcommand{\bpm}{\begin{pmatrix}}
\newcommand{\epm}{\end{pmatrix}}

\newcommand{\btab}{\begin{tabular}}
\newcommand{\etab}{\end{tabular}}

\newcommand{\mA}{\mathcal{A}}
\newcommand{\mB}{\mathcal{B}}

\newcommand{\mF}{\mathcal{F}}
\newcommand{\mH}{\mathcal{H}}

\newcommand{\mT}{\mathcal{T}}

\newcommand{\A}{\mathcal{A}}
\newcommand{\B}{\mathcal{B}}
\newcommand{\E}{\mathcal{E}}
\newcommand{\T}{\mathcal{T}}

\newcommand{\bR}{\mathbb{R}}

\newcommand{\oh}{{\otimes h}}
\newcommand{\cH}{\mathcal{H}}

\newcommand{\bC}{\mathbb{C}}
\newcommand{\cY}{\mathcal{Y}}
\newcommand{\cZ}{\mathcal{Z}}

\newcommand{\N}{\mathbb{N}}

\newcommand{\cpx}{\mathbb{C}}
\newcommand{\F}{\mathbb{F}}
\newcommand{\lmd}{\lambda}
\newcommand{\dt}{\delta}
\newcommand{\af}{\alpha}

\newcommand{\reff}[1]{(\ref{#1})}
\newcommand{\mc}[1]{\mathcal{#1}}
\newcommand{\supp}[1]{\mbox{supp}(#1)}

\newcommand{\baray}{\begin{array}}
\newcommand{\earay}{\end{array}}

\DeclareMathOperator{\rank}{rank}

\DeclareMathOperator{\hrank}{hrank}
\DeclareMathOperator{\brank}{brank}
\DeclareMathOperator{\hbrank}{hbrank}
\def\re{\mathbb{R}}
\def\cpx{\mathbb{C}}

\def\eps{\epsilon}

\def\nn{\nonumber}

\begin{document}

\title{Hermitian Tensor Decompositions}

\author{Jiawang Nie}
\address{Department of Mathematics,
University of California San Diego,
9500 Gilman Drive, La Jolla, CA, USA, 92093.}
\email{njw@math.ucsd.edu, ziy109@ucsd.edu}

\author{Zi Yang}


\date{}


\begin{abstract}
Hermitian tensors are generalizations of Hermitian matrices, 
but they have very different properties.
Every complex Hermitian tensor is a sum of complex Hermitian rank-$1$ tensors.
However, this is not true for the real case.
We study basic properties for Hermitian tensors
such as Hermitian decompositions and Hermitian ranks.
For canonical basis tensors, we determine their Hermitian ranks and decompositions.
For real Hermitian tensors, we give a full characterization for them
to have Hermitian decompositions over the real field.
In addition to traditional flattening,
Hermitian tensors specially have Hermitian and Kronecker flattenings,
which may give different lower bounds for Hermitian ranks.
We also study other topics such as eigenvalues,
positive semidefiniteness, sum of squares representations, and separability.
\end{abstract}

\maketitle

\section{Introduction}

Let $\F = \cpx$ (the complex field) or $\re$ (the real field).
For positive integers $m>0$ and $n_1, \ldots, n_m >0$,
denote by $\F^{n_1\times \cdots \times n_m}$ the space of
tensors of order $m$ and dimension $(n_1,\ldots, n_m)$ with entries in $\F$.
A tensor $\mA \in \F^{n_1\times\cdots \times n_m}$
can be represented as a multi-array
$\mA = (\mathcal{A}_{i_1...i_m} )$, with $i_k\in\{1,...,n_k\}$ for $k =1, \ldots, m$.
When $m=3$ (resp., $4$), they are called cubic (resp., quartic) tensors.
For vectors $u_k \in \F^{n_k}$, $k=1,\ldots,m$, the
$u_1 \otimes \cdots \otimes u_m$ denotes their tensor product, i.e.,
$
(u_1 \otimes \cdots \otimes u_m)_{ i_1 \ldots i_m}  =
(u_1)_{i_1}  \cdots (u_m)_{i_m}
$
for all $i_1, \ldots, i_m$ in the range.
Tensors like $u_1 \otimes \cdots \otimes u_m$ are called rank-1 tensors.
The {\it cp rank} of $\mA$, denoted as $\rank(\mA)$,
is the smallest $r$ such that
\be  \label{cpd:mA}
\mA = {\sum}_{i=1}^r u_i^1 \otimes \cdots \otimes u_i^m,
\quad u_i^j \in \cpx^{n_j} .
\ee
In the literature, the decomposition \reff{cpd:mA} is often called
a candecomp-parafac or canonical polyadic (CP) decomposition.
We refer to \cite{dLMV04,KolBad09,Land12,Lim13,tensorlab}
for tensor decompositions,
and refer to \cite{BreVan18,dLMV04,dLa06,SvBdL13}
for tensor decomposition methods.
For uniqueness of tensor decompostions,
we refer to the work~\cite{ChOtVan17,
domanov15,GalMel,Kru77,SB00}.
%
%

Symmetric matrices are natural generalizations of symmetric tensors.
A tensor $\mA \in \F^{n \times \cdots \times n}$ of order $m$
is {\it symmetric} if $\mA_{i_1 \ldots i_m}$ is invariant
for all permutations of $(i_1, \ldots, i_m)$. Rank-1 symmetric tensors are multiples of
$u^{\otimes m} := u\otimes \cdots \otimes u$ (repeated $m$ times).
Similarly, the smallest number $r$ such that
$
\A = {\sum}_{i=1}^r \lmd_i u_i^{\otimes m},
$
with each $u_i \in \cpx^n$ and $\lmd_i \in \cpx$,
is called the {\it symmetric rank} of $\A$.
We refer to \cite{BCMT10,Comon2008,Nie-GP,OedOtt13} for the work on
symmetric tensor decompositions. Symmetric tensors
can be generalized to partial symmetric tensors \cite{Land12} and
conjugate partial symmetric tensors \cite{FuJiangLi18}.
A class of interesting symmetric tensors are Hankel tensors \cite{NieYe19}.
More work about tensor ranks can be found in
\cite{CLQY18,YeLim18}.

Hermitian tensors are natural generalizations of Hermitian matrices,
while they have very different properties.
This concept was introduced by Ni~\cite{Ni19}.
For an array $u$, we use $\overline{u}$ to denote the complex conjugate of $u$.
A tensor $\mH \in \cpx^{n_1\times\cdots\times n_m\times n_1\times\cdots \times n_m}$
is called {\it Hermitian} if
\[
\mathcal{H}_{i_1...i_m j_1...j_m}=  \overline{\mathcal{H}_{j_1...j_m i_1...i_m} }
\]
for all labels $i_1, ...,  i_m$ and $j_1, ..., j_m$ in the range.
The set of all Hermitian tensors in
$\cpx^{n_1\times\cdots\times n_m\times n_1\times\cdots \times n_m}$
is denoted as $\bC^{[n_1, \ldots,  n_m]}$. 
Clearly, for vectors $v_i \in \cpx^{n_i}$, $i=1,\ldots,m$,
the following tensor product of conjugate pairs
\be \label{Eq:vectorproduct}
 [v_1, v_2,\ldots, v_m]_\oh \, := \,  v_1 \otimes v_2 \cdots \otimes v_m \otimes
\overline{v_1} \otimes \overline{v_2} \cdots \otimes \overline{v_m}
\ee
is always a Hermitian tensor. Every rank-$1$ Hermitian tensor
must be in the form of
$\lmd \cdot [v_1, v_2,\ldots, v_m]_\oh$, for a real scalar $\lmd \in \re$.
Every Hermitian matrix is a sum of
Hermitian rank-$1$ matrices, by spectral decompositions.
The same result holds for Hermitian tensors over the complex field.
For every $\mH \in \bC^{[n_1, \ldots,  n_m]}$,
Ni~\cite{Ni19} showed that there exist vectors
$u_i^j\in \mathbb{C}^{n_j}$ and real scalars $\lmd_i\in \mathbb{R}$,
$i=1,\ldots,r$, such that
\be \label{Eq:rank-oneHD}
\mathcal{H}= {\sum}_{i=1}^r \lmd_i \,
[u_i^1, \ldots, u_i^m]_{\otimes_h}.
\ee
The equation~\reff{Eq:rank-oneHD}
is called a {\it Hermitian decomposition}.
The smallest $r$ in \reff{Eq:rank-oneHD}
is called the {\it Hermitian rank} of $\mathcal{H}$,
for which we denote $\hrank(\cH)$.
When $r$ is minimum, \reff{Eq:rank-oneHD}
is called a {\it Hermitian rank decomposition} for $\mH$.
The set $\bC^{[n_1, \ldots,  n_m]}$ is a vector space over $\re$.
For its canonical basis tensors, we determine their Hermitian ranks
as well as the rank decompositions in the subsection \ref{subsc:bht}.
For general Hermitian tensors, it is a computational challenge
to determine their Hermitian ranks.

For two tensors $\mA, \mB \in \bC^{[n_1, \ldots,  n_m]}$,
their \emph{inner product} is defined as
\be   \label{Eq:tensorinnerproduct}
\langle\mathcal{A},\mathcal{B}\rangle :=
{\sum}_{i_1, \ldots, i_m, j_1, \ldots, j_m}
\mathcal{A}_{i_1 \ldots i_m j_1 \ldots j_m}
\overline{ \mathcal{B}_{i_1 \ldots i_m j_1 \ldots j_m} } .
\ee
The \emph{Hilbert-Schmidt norm} of $\mathcal{A}$ is accordingly defined as
$||\mathcal{A}||  \, := \, \sqrt{\langle\mathcal{A},\mathcal{A}\rangle}$.
If $\mA, \mB$ are Hermitian, then
$\langle\mathcal{A},\mathcal{B}\rangle $ is real \cite{Ni19}. 
For convenience of operations, we define multilinear matrix multiplications
for tensors (see \cite{Lim13}). 
For matrices $M_k \in \mathbb{C}^{p_k \times q_k }$,
$k=1,\ldots, m$, define the matrix-tensor product
$(M_1, \ldots, M_m) \times \mT$ for $\mT \in \cpx^{q_1 \times \cdots \times q_m}$
such that it gives a linear map from
$\cpx^{q_1 \times \cdots \times q_m}$ to
$\cpx^{p_1 \times \cdots \times p_m}$
and it satisfies
\[
(M_1, \ldots, M_m) \times (u_1 \otimes \cdots \otimes u_m) =
(M_1 u_1) \otimes \cdots \otimes (M_m u_m),
\]
for all rank-$1$ tensors $u_1 \otimes \cdots \otimes u_m$.
The product $(M_1, \ldots, M_m) \times \mT$ is a tensor in
$\cpx^{p_1 \times \cdots \times p_m}$.
For two tensors $\mT_1, \mT_2$ of compatible dimensions, it holds that
\[
\langle (M_1, \ldots, M_m) \times \mT_1, \mT_2 \rangle \, = \,
\langle \mT_1, (M_1^*, \ldots, M_m^*) \times \mT_2\rangle.
\]
(The superscript $^*$ denotes the conjugate transpose.)
For square matrices $Q_k \in \cpx^{n_k\times n_k}$, $k=1,\ldots, m$,
we define the {\it multilinear congruent transformation}
for $\mA \in \bC^{[n_1, \ldots,  n_m]}$ such that
\begin{equation} \label{prod:cong}
(Q_1, \ldots, Q_m) \times_{cong} \mA  \, := \,
(Q_1, \ldots, Q_m, \overline{Q_1}, \ldots, \overline{Q_m} ) \times \mA.
\end{equation}
If each $Q_k$ is unitary, then
$\mB := (Q_1, \ldots, Q_m) \times_{cong} \mA$
is called a {\it unitary congruent transformation} of $\mA$
and $\mB$ is said to be {\it unitarily congruent} to $\mA$.
%
%
It holds that
\[
(Q_1^*, \ldots, Q_m^*) \times_{cong}
\Big( (Q_1, \ldots, Q_m) \times_{cong} \mA \Big) = \mA.
\]
If each $Q_k$ is real and orthogonal,
the tensor $\mathcal{B}$ is said to be {\it orthogonally congruent} to $\mA$.
Unitary and orthogonal congruent transformations
preserve norms of Hermitian tensors \cite{Ni19}.

%
%

Hermitian tensors have important applications in quantum physics \cite{Ni19}.
An $m$-partite pure state $|\psi \rangle $ of a quantum system can be represented
by a tensor in $\bC^{n_1\times \cdots \times n_m}$.
The complex conjugate of $|\psi \rangle $ represents another pure state $\langle \psi |$.
The conjugate product $|\psi \rangle \langle \psi|$ represents a
$2m$-partite pure state in the Hermitian tensor space $\bC^{[n_1, \ldots, n_m]}$.
A mixed quantum state can be represented by a Hermitian tensor.
The state is called unentangled (or separable) if it can be expressed as a sum of
rank-$1$ pure state products like $|\psi \rangle \langle \psi|$;
otherwise, the state is called entangled (or not separable).
%
%
Equivalently, a mixed state $\rho \in \bC^{[n_1, \ldots, n_m]}$
is unentangled if and only if
\[
\rho \,=\, \sum_{i=1}^k | \psi_i \rangle \langle \psi_i |
\]
for some rank-$1$ pure sates $| \psi_i \rangle$.
%
Mathematically, the above is equivalent to the Hermitian decomposition
\[
\rho \, = \, \sum_{i=1}^k  (u_i^1\otimes \cdots \otimes u_i^m)\otimes
\overline{(u_i^1\otimes \cdots \otimes u_i^m)}
\,= \, \sum_{i=1}^k  [u_i^1, \ldots , u_i^m]_\oh ,
\]
for complex vectors $u_i^1 \in \cpx^{n_1}, \ldots, u_i^m \in \cpx^{n_m}$.
Hermitian tensors, which can be decomposed as above, are called separable tensors.
%
%
Hermitian tensors representing mixed states are also called
density matrices. In view of algebra, Hermitian tensors can also be regarded as
real valued complex conjugate polynomials. Detection of unentangled mixed states
is related to separability of Hermitian tensors. We refer to
\cite{ardila2018measuring,blum2012density,calderaro2018direct,DLMO07}
for applications of density matrices.
%
%
Quantum information theory is closely related to tensors
\cite{DFLW17,LNSU18,NQB14,Ni19,QiZhangNi18}.
The separability issue will be studied in section~\ref{sc:sepent}.

%
%
\medskip \noindent
{\bf Contributions}\,
The paper studies Hermitian tensors.
They have very different properties from the matrix case.
For each canonical basis tensor of $\bC^{[n_1, \ldots, n_m]}$,
we determine the Hermitian rank, as well as the rank decomposition.
After that, we present some general properties about Hermitian decompositions and Hermitian ranks.
This is given in Section~\ref{sc:R1HTD}.

Every complex Hermitian tensor is a sum of complex Hermitian rank-$1$ tensors.
However, this is not true for the real case.
A real Hermitian tensor may not be able to be written as a sum of
real Hermitian rank-$1$ tensors.
We give a full characterization for real Hermitian tensors to
have real Hermitian decompositions.
Interestingly, the set of real Hermitian decomposable tensors
form a proper subspace.
The relationship between real and complex Hermitian decompositions
are also discussed.
This is presented in Section~\ref{sc:real}.

For Hermitian tensors, there are two special types of matrix flattening,
i.e., the Hermitian flattening and Kronecker flattening,
in addition to traditional flattening.
The Hermitian and Kronecker flattenings
may provide different lower bounds for Hermitian ranks.
Some new decompositions can also be obtained
from the Hermitian flattening.
This is shown in Section~\ref{sc:MatFlat}.

Positive semidefinite (psd) Hermitian tensors are also investigated.
They can be characterized by sum of squares (SOS) decompositions.
There are two different types of SOS decompositions, i.e.,
the Hermitian SOS and conjugate SOS decompositions.
They can be used to characterize psd Hermitian tensors.
Hermitian eigenvalues can also be applied to do that.
This is discussed in Section~\ref{sc:nhe}.

We also study separable Hermitian tensors, which
can be written as sums of Hermitian tensors in the form $[v_1,\ldots, v_m]_\oh$.
Separable Hermitian tensors can be characterized
in terms of truncated moment sequences
or its Hermitian flattening matrix decompositions.
Interestingly, the cone of separable Hermitian tensors
is dual to the cone of psd Hermitian tensors.
This is done in Section~\ref{sc:sepent}.

The paper is concluded in Section~\ref{sc:con},
with a list of some open questions for future work.

\medskip
\noindent
{\bf Notation} \,
The $\N$ denotes the set of nonnegative integers.
For $k=1,\ldots,m$, the $x_k$ denotes the
complex vector variable in $\cpx^{n_k}$.
The tuple of all such complex variables is denoted as
$x:=(x_1,\ldots, x_m)$. For $\F = \re$ or $\cpx$,
denote by $\F[x]$ the ring of polynomials in $x$
with coefficients in $\F$,
while $\F[x,\overline{x}]$ denotes the ring of conjugate polynomials in
$x$ and $\overline{x}$ with coefficients in $\F$.
In the Euclidean space $\F^n$,
denote by $e_i$ the $i$th standard unit vector, i.e.,
the $i$th entry of $e_i$ is one and all others are zeros,
while $e$ stands for the vector of all ones.
The $I_k$ denotes the $k$-by-$k$ identity matrix.
For a vector $u$ in $\re^n$ or $\cpx^n$,
$\| u \|$ denotes its standard Euclidean norm.
For a matrix or vector $a$, the $a^*$ denotes its conjugate transpose,
$a^T$ denotes its transpose, while $\overline{a}$ denotes its conjugate entry wise;
we use $\mbox{Re}(a)$ and $\mbox{Im}(a)$
to denote its real and complex part respectively.
For a complex scalar or vector $z$, denote $|z| := \sqrt{z^*z}$.
The $\mbox{int}(S)$ denotes the interior of a set $S$,
under the Euclidean topology.
The $\mathbb{M}^n$ denotes the set of $n$-by-$n$ Hermitian matrices,
while $\mc{S}^n$ denotes the set of
$n$-by-$n$ real symmetric matrices.
If a Hermitian matrix $X$ is positive semidefinite (resp., positive definite),
we write that $X \succeq 0$ (resp., $X \succ 0$).
The symbol $\otimes$ denotes the tensor product,
while $\boxtimes$ denotes the classical Kronecker product.
For a tensor product $u \otimes v \otimes \cdots$,
we denote by $\mbox{vec}(u \otimes v \otimes \cdots)$
the column vector of its coefficients in its representation
in terms of the basis tensors.
For an integer $k>0$, denote the set $[k] := \{1,\ldots, k\}$.
For a real number $t$, the ceiling $\lceil t \rceil$
denotes the smallest integer that is greater than or equal to $t$.

\section{Hermitian decompositions and ranks}
\label{sc:R1HTD}

This section studies Hermitian decompositions and ranks. Hermitian decompositions
can be equivalently expressed by conjugate polynomials.
For complex vector variables $x_k \in \cpx^{n_k}$,
$k=1,\ldots,m$, denote $x:=(x_1,\ldots, x_m)$.  The inner product
\[
\cH(x,\overline{x}) :=   \langle \cH, [x_1, \ldots, x_m]_\oh \rangle
\]
is a conjugate symmetric polynomial in $x$, i.e.,
$\cH(x,\overline{x})= \overline{ \cH(x,\overline{x}) }$. It only achieves real values
\cite{JiangLiZh16,Ni19}. The decomposition
$\cH = \sum_{i=1}^r \lmd_i [u_i^{1}, \ldots, u_i^{m}]_\oh$
is equivalent to the polynomial decomposition
\be \label{H(x):hermSq}
\cH(x,\overline{x})
={\sum}_{i=1}^r \lmd_i |(u_i^1)^*x_1|^2 \cdots |(u_i^m)^*x_m|^2 .
\ee
Therefore, a Hermitian decomposition of $\cH$ can be equivalently expressed as
a real linear combination of conjugate squares
like $|(u_i^1)^*x_1|^2 \cdots |(u_i^m)^*x_m|^2$.

\subsection{Hermitian decompositions for basis tensors}
\label{subsc:bht}

For convenience, denote
\[
N \,:= \, n_1 \cdots n_m, \quad
\mc{S}  := \, \Big\{ (i_1, \ldots, i_m): \,
 i_1\in [n_1], \ldots, i_m\in [n_m]  \Big \}.
\]
The cardinality of the label set $\mc{S}$ is $N$.
For two labelling tuples $I :=(i_1, \ldots, i_m)$ and $J :=(j_1, \ldots, j_m)$
in $\mc{S}$, define the ordering $ I < J $
if the first nonzero entry of $I-J$ is negative.
For a scalar $c \in \cpx$,
denote by $\mathcal{E}^{IJ}(c)$ the Hermtian tensor
in $\bC^{[n_1, \ldots,  n_m]}$ such that
\[
\big( \mathcal{E}^{IJ}(c) \big)_{i_1\cdots i_m j_1\cdots j_m}=
\overline{ \big(\mathcal{E}^{JI}(c) \big)}_{j_1\cdots j_m i_1\cdots i_m}=c
\]
and all other entries are zeros.
%
%
We adopt the standard scalar multiplication and addition
for $\bC^{[n_1,\ldots,n_m]}$,
so $\bC^{[n_1,\ldots,n_m]}$ is a vector space over $\mathbb{R}$.
The set
\be \label{basis:E}
E \, := \, \Big \{ \mathcal{E}_{II}(1) \Big \}_{ I \in \mc{S} } \bigcup
\Big \{
\mathcal{E}_{I J}(1),
\mathcal{E}_{I J}(\sqrt{-1})
\Big \}_{ I, J \in \mc{S}, I< J  }
\ee
is the {\it canonical basis} for
$\bC^{[n_1,\ldots, n_m]}$. Its dimension is
\[
\dim \bC^{[n_1,\ldots,n_m]} = N + N(N-1) = N^2.
\]
For these basis tensors, we determine their Hermitian ranks
as well as the rank decompositions.
For a basis tensor $\mc{E}^{IJ}(c)$, we are interested in $c=1$ or $\sqrt{-1}$.
Its Hermitian rank can be determined by reduction to the $2$-dimensional case.

\begin{lemma} \label{thm: basis rank dim 2}
Suppose the dimensions $n_1,\ldots, n_m \geq 2$,
$I={(i_1,\ldots,i_m)}$, and $J={(j_1,\ldots,j_m)}$.
For each $k=1,\ldots,m$, let
\[
(i_k',j_k') := (1,1)  \quad \text{if $i_k=j_k$}, \quad
(i_k',j_k') := (1,2)  \quad \text{if $i_k\neq j_k$}.
\]
Let $I' :=(i_1',\ldots,i_m'), J' := (j_1',\ldots,j_m') $.
Then, $\E^{I'J'}(c) \in \bC^{[2,\ldots,2]} $ and
\[
\hrank \, \E^{IJ}(c)  \,= \, \hrank \, \E^{I'J'}(c).
\]
\end{lemma}
\begin{proof}
For each $k$, if $i_k = j_k$,
let $P_k$ be the permutation matrix that switches
the 1st and $i_k$th rows; if $i_k \ne j_k$,
let $P_k$ be the permutation matrix that
switches $i_k$th row and $j_k$th row to $1$st row and $2$nd row respectively.
Consider the orthogonal congruent transformation
\be \nn
\mc{F}  \, := \, (P_1,\ldots,P_m)\times_{cong} \E^{IJ}(c) .
\ee
Then $\mc{F}$ is the Hermitian tensor such that
$
\mc{F}_{I' J'} = \overline{ \mc{F}_{J' I'} } = c
$
and all other entries are zeros, so
$\mc{F}$ is a canonical basis tensor. Note that
$\E^{I'J'}(c)$ is the subtensor of $\mc{F}$,
consisting of the first two labels for each dimension,
hence $\E^{I'J'}(c)$ and $\mc{F}$ have the same rank.
Since nonsingular congruent transformations
preserve Hermitian ranks (see Proposition~\ref{prop:rankQA=A}),
$
\hrank \, \E^{IJ}(c)  =  \hrank \, \E^{I'J'}(c).
$
\end{proof}

In the following, for $n_1=\cdots=n_m=2$
and $I=(1 \ldots 1)$, $J=(2 \ldots 2)$,
we determine the Hermitian rank of
the basis tensor $\E^{IJ}(c)$. First, we consider $c=1$.
For each $k=0,1,\ldots,m$, let
\be \label{thetak:uk}
\theta_k := {k \pi}/{m},  \quad
u_k := (1, \exp\big( \theta_k \sqrt{-1} ) \big).
\ee
%
%
The following Hermitian tensor
\be \label{tensor:Ak}
 \A_k := \frac{1}{2}\big([u_k,u_k,\ldots,u_k]_\oh
 + [\overline{u_k},\overline{u_k},\ldots,\overline{u_k}]_\oh \big)
\ee
has rank $1$ or $2$. For each $s= 0,1, \ldots, m$,
let $J_s: =(1,\ldots,1, 2,\ldots,2)$ where $2$ appears $s$ times.
%
%
The tensor $\A_k$ has only $m+1$ distinct entries, which are
\[
 (\A_k)_{IJ_s}=\mbox{Re}\big( (u_k)_2^s \big)=
 \mbox{Re}\big( \exp(s\theta_k\sqrt{-1}) \big) = \cos(s\theta_k), \quad s= 0,1, \ldots, m.
\]
For each $k$, consider the vector
\[
w_k \, := \, \big(\cos(0 \cdot \theta_k),\cos(1 \cdot \theta_k),
\ldots, \cos(m \cdot \theta_k) \big) .
\]
Let $\lambda_k := 2(-1)^k$ for $1\le k\le m-1$,
$\lambda_k:=(-1)^k $ for $k=0,m$, and
\be \label{vec:u:sum:wk}
u  := \lambda_0 w_0+\lambda_1w_1+\cdots + \lambda_m w_m .
\ee
For $p=0,1,\ldots,m$, the $(p+1)$th entry of $u$ is
\[
(u)_{p+1} = {\sum}_{k=0}^m \lambda_k \cos(p\theta_k) =
{\sum}_{k=0}^m \lambda_k \cos(\frac{pk}{m }\pi) =
\mbox{Re}\Big( {\sum}_{k=0}^m \lambda_k \exp( \frac{pk}{m}\pi \sqrt{-1}) \Big) .
\]
For each $p=0,1,\ldots,m-1$, one can check that
(let $\af : = \frac{p}{m}\pi$ )
\[
{\sum}_{k=0}^m \lambda_k \exp( k \af \sqrt{-1} ) =
2{\sum}_{k=0}^m (-1)^k \exp( k \af \sqrt{-1} )
 - 1 - (-1)^m \exp( p\pi \sqrt{-1} )
\]
\[
= 2 \frac{1-(-\exp((m+1)\alpha \sqrt{-1})}{1+\exp( \alpha \sqrt{-1} )}- 1 - (-1)^{m+p}
\]
\[
=\bca
  0 & \text{if $m+p$ is even},  \\
 \frac{-4\sin \alpha }{(1+\cos \alpha)^2+(\sin \alpha)^2} \sqrt{-1} & \text{if $m+p$ is odd}.
\eca
\]
Hence, $(u)_{p+1} = 0$ for $0\le p \le m-1$. Moreover,
\[
(u)_{m+1} = {\sum}_{k=0}^m \lambda_k \cos(m\theta_k)
= {\sum}_{k=0}^m \lambda_k \cos(k\pi)
= {\sum}_{k=0}^m \lambda_k (-1)^k = 2m .
\]
Therefore, we have
\[
{\sum}_{k=0}^m \frac{\lambda_k}{2m} w_k = (0,\ldots,0,1), \quad
\E^{IJ}(1) = {\sum}_{k=0}^m \frac{\lambda_k}{2m}\A_k .
\]
This gives the Hermitian decomposition of length $2m$:
\begin{multline}  \label{HTD:E1122:c=1}
\E^{IJ}(1) = \frac{1}{2m} \Big([u_0,u_0,\ldots,u_0]_\oh
+(-1)^m[u_m,u_m,\ldots,u_m]_\oh
 \\
 + {\sum}_{k=1}^{m-1} (-1)^k([u_k,u_k,\ldots,u_k]_\oh+
 [\overline{u_k},\overline{u_k},\ldots,\overline{u_k}]_\oh)
\Big),
\end{multline}
where $u_k$ is given as in \reff{thetak:uk}.
For the case $c \ne 0$, one can verify that
\[
\E^{IJ}(c) = ( \bpm c & 0 \\ 0  & 1 \epm,
\bpm 1 & 0 \\ 0  & 1 \epm, \ldots, \bpm 1 & 0 \\ 0  & 1 \epm ) \times_{cong}  \E^{IJ}(1).
\]
Then, the decomposition~\reff{HTD:E1122:c=1} implies that
\begin{multline}  \label{rkHTD:E1122c}
\E^{IJ}(c) = \frac{1}{2m} \Big([\tilde{u}_0,u_0,\ldots,u_0]_\oh+(-1)^m[\tilde{u}_m,u_m,\ldots,u_m]_\oh
\\
+ {\sum}_{k=1}^{m-1} (-1)^k([\tilde{u}_k,u_k,\ldots,u_k]_\oh
+[{\tilde{v}_k},\overline{u_k},\ldots,\overline{u_k}]_\oh)
\Big),
\end{multline}
where $\tilde{u}_k =\big(c , \exp(\frac{k}{m} \pi \sqrt{-1}) \big)$
and $\tilde{v}_k = \big(c , \exp( -\frac{k}{m} \pi \sqrt{-1}) \big)$.
%
%



\begin{prop} \label{prop:1122 rank}
Assume $n_1 = \cdots = n_m = 2$, $I=(1 \ldots 1)$, $J=(2 \ldots 2)$,
and $c \ne 0$. Then, $\hrank(\E(c))=2m $
and \reff{rkHTD:E1122c} is a Hermitian rank decomposition.
\end{prop}
\begin{proof}
The decomposition~\reff{rkHTD:E1122c} implies
$\hrank(\E^{IJ}(c)) \leq 2m$, so we only need to show
$\hrank(\E^{IJ}(c)) \geq 2m$.
We prove it by induction on $m$.

When $m=1$, $\E^{(12)}(c)$ is a Hermitian matrix of rank $2$
and the conclusion is clearly true.
Suppose the conclusion holds for $m=1,2,\ldots,k$.
Assume to the contrary that for $m=k+1$, $r:=\hrank(\E^{IJ}(c))\le 2m-1=2k+1$
and $\E^{IJ}(c)$ has the Hermitian decomposition (for nonzero vectors $u_i^j$):
\[
 \E^{IJ}(c) = {\sum}_{i=1}^{r}\lambda_i[u^1_i,\ldots,u^{k+1}_i]_\oh.
\]
Let $\A_i = \lambda_i[u^1_i,\ldots,u^{k}_i]_\oh$,
$U_i = u^{k+1}_i \otimes \overline{u^{k+1}_i} $,
then $\E^{IJ}(c)$ can be rewritten as (after a reordering of tensor products)
\[
 \E^{IJ}(c) = {\sum}_{i=1}^{r} \A_i\otimes U_i.
\]
Let $p$ be the dimension of $\mbox{span}\{U_1,\ldots,U_r\}$
and one can generally assume $\{U_1,\ldots,U_p\} $ is linearly independent.
Then $U_j = \sum_{s=1}^p \alpha^j_s U_s$, $j>p$,
for some real coefficients $\alpha^j_s$,
since each $U_i$ can be viewed as a Hermitian matrix.
So we can rewrite that
\[
\E^{IJ}(c)
 = {\sum}_{i=1}^p \B_i\otimes U_i
\quad \mbox{where} \quad
\B_i :=\A_i+ {\sum}_{j=p+1}^r \alpha_i^j \A_j .
\]
Each $\B_i$ is a Hermitian tensor of order $2k$, and $\hrank(\B_i)\le r-p+1 $.
For two labels $I',J' \in \N^k$, consider the matrix
\[
M^{I'J'}:=
\bbm
(\E^{IJ}(c))_{(I',1)(J',1)}  & (\E^{IJ}(c))_{(I',1)(J',2)} \\
(\E^{IJ}(c))_{(I',2)(J',1)}  & (\E^{IJ}(c))_{(I',2)(J',2)}
\ebm = \sum_{i=1}^p (\B_i)_{I'J'} U_i .
\]
Note that $M^{I'J'} \neq 0 $
if and only if $I'=(1\cdots 1),J'=(2\cdots 2) $ or
$I'=(2\cdots 2),J'=(1\cdots 1)$.
Since $U_1,\ldots,U_p$ are linearly independent,
$((\B_1)_{I'J'}, \ldots, (\B_p)_{I'J'}) \ne 0$
if and only if $I'=(1\cdots 1),J'=(2\cdots 2) $ or
$I'=(2\cdots 2),J'=(1\cdots 1)$. So each nonzero
$\B_i$ is also a canonical basis tensor in $\bC^{[2,\ldots,2]} $.
By induction, we have
\[
  r-p+1 \ge \hrank(\B_i) \ge 2k, \qquad p\le r+1-2k \le 2 .
\]

By the same argument, we can show that the rank of the set
$V_j := \big\{u^{j}_i\otimes \overline{u^{j}_i} \big\}_{i=1}^{r} $ is at most $2$,
for all $j=1,\ldots,m$. If the rank of $V_j$ is $2$,
then there exists $ t_j \in [r]$ such that
$\{u^{j}_1\otimes \overline{u^{j}_1},u^{j}_{t_j}\otimes \overline{u^{j}_{t_j}}\} $
is linearly independent. If the rank of $V_j$ is $1$,
we let $t_j: = 1$. Thus
$u^{j}_i= u^{j}_1$ or $u^{j}_i=u^{j}_{t_j} $ for each $i=1,\ldots,r$.
For each $j$, there exists $w^j$ such that $(w^j)^T\overline{u^{j}_1}=1$, and $ (w^j)^T\overline{u^{j}_{t_j}}=0 $ if $t_j > 1$.
Then, consider the multilinear matrix-tensor product
\[
 \T:= (I_2,\ldots,I_2, (w^1)^T,\ldots, (w^{k+1})^T ) \times \E^{IJ}(c)
 = \lambda_1 u^{1}_1 \otimes \cdots \otimes u^{k+1}_1 \in \bC^{2\times \cdots \times 2}.
\]
When $(s_1\cdots s_{k+1}) \neq (1,\ldots,1) $ or $(2,\ldots, 2) $, we have
\[
	\T_{s_1\cdots s_{k+1}}= \sum_{j_1,\ldots, j_{k+1} = 1, 2}(w^1)_{j_1} \cdots (w^{k+1})_{j_{k+1}} (\E^{IJ}(c))_{s_1\ldots s_{k+1} j_1 \ldots j_{k+1}} = 0.
\]
So $\T$ has at most two nonzero entries,
which must be $\T_{1\cdots 1}$ and/or $\T_{2\cdots 2}$:
\[
\baray{l}
\T_{1\cdots 1} = (\E^{IJ}(c))_{(1\cdots 1)(2\cdots 2)} (w^1)_2\cdots (w^{k+1})_2
= c (w^1)_2\cdots (w^{k+1})_2, \\
\T_{2\cdots 2} = (\E^{IJ}(c))_{(2\cdots 2)(1\cdots 1)} (w^1)_1\cdots (w^{k+1})_1
= \overline{c}(w^1)_1\cdots (w^{k+1})_1.
\earay
\]
Since $\T$ is rank $1$, only one of $\T_{1\cdots 1},\T_{2\cdots 2} $
is nonzero, which is also the unique nonzero entry of $\T$.
Without loss of generality, assume $\T_{1\cdots 1}\neq 0,\T_{2\cdots 2}=0 $.
The fact that $(\T)_{1\cdots 1} $ is the only one nonzero entry implies
$u^j_1 = \mu_j e_1,j=1\cdots k+1$ for some $0\neq  \mu_j \in \bC $.
The equation $(w^j)^T\overline{u^j_1} = \overline{\mu}_j(w^j)_1 =1$
implies that $(w^j)_1 \neq 0$,  so
$\T_{2\cdots 2}=\overline{c}(w^1)_1 \cdots (w^{k+1})_1\neq 0$.
But this contradicts $\T_{2 \ldots 2} = 0$,
hence $\hrank(\E^{IJ}(c))\ge 2m$.
\end{proof}

Ranks of basis tensors $\E^{IJ}(c)$ for general dimensions
are given as follows.

\begin{theorem} \label{thm:basisHR}
Assume $n_1,\ldots, n_m \geq 2$,
$I={(i_1,\ldots,i_m)},J={(j_1,\ldots,j_m)}$, and $c \ne 0$.
If $I=J$, then $\hrank \E^{IJ}(c) = 1$;
if $I \ne J$, then $\hrank \E^{IJ}(c) = 2d$
where $d$ is the number of nonzero entries of $I-J$.
\end{theorem}
\begin{proof}
When $I=J$, $\E^{IJ}(c)$ is a Hermitian tensor only if $c$ is real, and
$\E^{II}(c) = c [e_{i_1},\ldots,e_{i_m}]_\oh$.
So, $\hrank \E^{II}(c) =1$.
When $I \ne J$, we can generally assume
$i_k\neq j_k$ for $k = 1, \ldots, d$,
and $i_k=j_k$ for $k = d+1, \ldots, m$. By Lemma~\ref{thm: basis rank dim 2},
$\E^{IJ}(c)$ has the same Hermitian rank as $\E^{I'J'}(c)$,
for $I'=(1,\ldots,1)$ and  $J'=(2,\ldots,2, 1,\dots,1)$
(the first $d$ entries of $J'$ are $2$'s).
%
%
Let $I_1 =(1,\ldots, 1), I_2 =(2, \ldots, 2)$,
where $1,2$ are repeated for $d$ times.
Then $\hrank \E^{I'J'}(c) = \hrank \E^{I_1 J_1}(c) $.
By Proposition~\ref{prop:1122 rank}, we know
$
 \hrank \E^{IJ}(c) = \hrank \E^{I_1 J_1}(c) = 2d.
$
\end{proof}

The following is an example of Hermitian rank decompositions for basis tensors.

\begin{example} \label{Ex:Herm Decom m=2}
For $I=(1,2)$, $J=(3,4)$ and $c \ne 0$,
the basis tensor $\E^{(12)(34)}(c)\in \mathbb{C}^{[4,4]}$
has the Hermitian rank $4$, with the
following Hermitian rank decomposition
(in the following $i :=\sqrt{-1}$) {\scriptsize
\[
\frac{1}{4} \left[ \begin{pmatrix}
            c\\0\\ 1 \\ 0
            \end{pmatrix},
            \begin{pmatrix}
            0\\1 \\0 \\ 1
            \end{pmatrix} \right]_\oh
+
          \frac{1}{4}  \left[ \begin{pmatrix}
            c\\0\\ -1 \\ 0
            \end{pmatrix},
            \begin{pmatrix}
            0\\1 \\0 \\ -1
            \end{pmatrix} \right]_\oh
-
          \frac{1}{4}  \left[ \begin{pmatrix}
            c\\0\\ i \\ 0
            \end{pmatrix},
            \begin{pmatrix}
            0\\1 \\0 \\ i
            \end{pmatrix} \right]_\oh
-
           \frac{1}{4} \left[ \begin{pmatrix}
            c\\0\\ -i \\ 0
            \end{pmatrix},
            \begin{pmatrix}
            0\\1 \\0 \\ -i
            \end{pmatrix} \right]_\oh.
\]}
\end{example}

\subsection{Basic properties of Hermitian decompositions}
\label{ssc:Hrank}

In some occasions, a Hermitian tensor may be given by
a Hermitian decomposition.
One wonders whether that is a rank decomposition or not.
This question is related to the classical Kruskal theorem \cite{Kru77,SB00}.
For a set $S$ of vectors, its {\it Kruskal rank},
denoted as $\mathit{k}_S$, is the maximum number $k$
such that every subset of $k$ vectors in $S$
is linearly independent.

\begin{prop} \label{th:kruskal}
Let $\cH = \sum_{j=1}^r \lmd_j [u_j^{1}, \ldots, u_j^{m}]_\oh$
be a Hermitian tensor, with $0 \ne \lmd_j \in \re $ and $m>1$.
For each $i=1,\ldots, m$, let $U_i := \{u_1^{i}, \ldots, u_r^{i} \}$. If
\be \label{sum:kS>=r+m}
\mathit{k}_{U_1} + \cdots + \mathit{k}_{U_m} \geq r+m,
\ee
then $\text{hrank}(\cH) = r$ and
the Hermitian rank decomposition of $\mH$ is essentially unique, i.e.,
it is unique up to permutation and scaling of decomposing vectors.
\end{prop}
\begin{proof}
Note that $\mathit{k}_{U_i} = \mathit{k}_{ \overline{U_i} }$,
where $\overline{U_i} := \{\overline{u_j^i}, \ldots, \overline{u_j^i} \}$.
The rank condition~\reff{sum:kS>=r+m} is equivalent to that
\[
\mathit{k}_{U_1} + \cdots + \mathit{k}_{U_m} +
\mathit{k}_{\overline{U_1}} + \cdots + \mathit{k}_{\overline{U_m}} \geq 2r+2m-1.
\]
The conclusion is then implied by the classical Kruskal type theorem~\cite{Kru77,SB00}
(or see Theorems~12.5.3.1 and 12.5.3.2 in \cite{Land12}).
\end{proof}

For instance, for the following vectors
\[
u_1 =(1,1,1), \, u_2=(1,1,0), \,
u_3=(1,0,1), \, u_4=(0,1,1),
\]
the sum $\sum_{i=1}^4 [u_i, u_i, u_i]_\oh$
has Hermitian rank $4$, by Proposition~\ref{th:kruskal}.
This is because, for $U=\{u_1,u_2,u_3,u_4\}$, the Kruskal rank
$k_U=3$, $m=3$ and $3k_U=9 \ge 4+m=7$.

A basic question is how to compute Hermitian rank decompositions.
This is generally a challenge.
When Hermitian ranks are small,
we can apply the existing methods for
canonical polyadic decompositions (CPDs) for cubic tensors.
For convenience, let
\be \label{val:N1N2N3}
N_1 := n_1 \cdots n_m, \quad N_3:=\min\{ n_1, \ldots, n_m \},  \quad
N_2= N_1/N_3.
\ee
Up to a permutation of dimensions,
we can assume $n_m$ is the smallest, i.e., $N_3 = n_m$.
A Hermitian tensor can be flattened to a cubic tensor.
Define the linear flattening mapping
$\psi: \bC^{[n_1,\ldots,n_m]} \to \cpx^{N_1 \times N_2 \times N_3}$ such that
\be  \label{df:map:varphi}
\psi( [u^1, \ldots, u^m]_\oh )   \, =  \,
(u^1 \otimes  \cdots \otimes  u^m) \otimes
(\overline{u^1}  \otimes \cdots \otimes \overline{u^{m-1}}) \otimes \overline{u^m}.
\ee
Then
$\cH = \sum_{j=1}^r \lmd_j [u^1_{j}, \ldots, u^m_{j}]_\oh$
if and only if
\be   \label{td:varphi(H)}
\psi(\cH) ={\sum}_{j=1}^r \lmd_j \, a_j \otimes b_j \otimes c_j
\ee
where
$
a_j = u_j^{1} \otimes \cdots \otimes u_j^{m}, \,
b_j = \overline{u_j^{1}} \otimes \cdots \overline{u_j^{m-1}}, \,
c_j = \overline{u_j^{m}}.
$
The decomposition \reff{td:varphi(H)}
can be obtained by computing the CPD for $\psi(\cH)$,
if the rank decomposition of $\psi(\cH)$ is unique.
We refer to \cite{BBCM13,BreVan18,dLMV04,dLa06,tensorlab} for computing CPDs.

\begin{example}
Consider the tensor $\A \in \cpx^{[3,3]} $ such that
$
\A_{i_1i_2j_1j_2} = i_1j_1+i_2j_2
$
for all $i_1,i_2,j_1,j_2$ in the range.
A Hermitian decomposition for $\A$ is
\[  
\A = \left[
 \begin{pmatrix}
	1 \\ 2 \\ 3
\end{pmatrix},
\begin{pmatrix}
	1 \\ 1 \\ 1
\end{pmatrix} \right]_\oh
+
\left[\begin{pmatrix}
	1 \\ 1 \\ 1
\end{pmatrix},
\begin{pmatrix}
	1 \\ 2 \\ 3
\end{pmatrix} \right]_\oh .
\]
By Proposition~\ref{th:kruskal}, the Hermitian rank is $2$. 
\end{example}

The rank of a Hermitian matrix does not change after a nonsingular
congruent transformation. The same conclusion holds for Hermitian tensors.
We refer to \reff{prod:cong} for multi-linear congruent transformations.

\begin{prop}  \label{prop:rankQA=A}
Let $Q_k \in \cpx^{n_k \times n_k}$
be nonsingular matrices, for $k=1,\ldots,m$. Then, for each
$\mH \in \bC^{[n_1,\ldots,n_m]}$, the congruent
transformation $(Q_1,\ldots, Q_m) \times_{cong} \mH$
has the same Hermitian rank as $\mH$ does.
\end{prop}
\begin{proof}
Let $\mF := (Q_1,\ldots, Q_m) \times_{cong} \mH$, then
$\mH ={\sum}_{i=1}^r \lmd_i [u_i^1,\ldots, u_i^m]_{\oh}$
if and only if
$
\mF ={\sum}_{i=1}^r \lmd_i [Q_1 u_i^1,\ldots, Q_m u_i^m]_{\oh},
$
because each $Q_i$ is nonsingular.
So $\hrank(\mH) = \hrank(\mF)$.
\end{proof}

\subsection{Border, expected, generic and typical ranks}
\label{ssc:hbr}

There exist classical notions of
border, expected, generic and typical tensor ranks \cite{Land12}.
They all can be similarly defined for Hermitian ranks.
The classical border rank of a tensor $\mA$ is the smallest $r$
such that $\mA_k \to \mA$, where each $\mA_k$ is a rank-$r$ tensor.
The border rank of $\mA$ is denoted as $\brank(\mA)$.
We can similarly define Hermitian border ranks.

\begin{definition}  \label{def:border:HR}
For $\cH \in \bC^{[n_1,\ldots,n_m]}$,
its Hermitian border rank, for which we denote $\hbrank(\cH)$,
is the smallest $r$ such that there is a sequence
$\{ \cH_k \}_{k=1}^\infty \subseteq \bC^{[n_1,\ldots,n_m]}$
such that $\cH_k \to \cH$ and each $\hrank(\cH_k) = r$.
\end{definition}

Like the classical border rank inequality, we also have
\be \label{border<=hrank}
\brank(\cH) \le \hbrank(\cH) \leq \hrank(\cH).
\ee
The strict inequality can occur, as shown in the following example.

\begin{example} \label{exm:bordrank=2}
Consider the Hermitian tensor $\mB$ that is given as
\[
e_2\otimes e_1 \otimes e_1 \otimes e_1
+e_1\otimes e_2 \otimes e_1 \otimes e_1
+e_1\otimes e_1 \otimes e_2 \otimes e_1
+e_1\otimes e_1 \otimes e_1 \otimes e_2.
\]
For each $k>0$, denote the rank-$2$ Hermitian tensor
\[
\mB_k =k([e_1+\frac{e_2}{k},e_1+\frac{e_2}{k}]_\oh-[e_1,e_1]_\oh) .
\]
Since $\mB_k \to \mB $, $\hbrank (\mB) \leq 2$.
This kind of tensors are investigated in \cite{DeSLim08}.
The border rank is less than the cp rank.
The Hermitian flattening matrix of $\mB$ is (see \reff{H=qq*:rk1}) has rank $2$.
So, $\hbrank (\mB) \ge 2$ by Lemma~\ref{lema:matrix rank},
and hence $\hbrank(\mB) = 2$. However, $\hrank(\A)=4$.
Note the decomposition
\[
\mB = \frac{1}{2} \big(
[e_1, e_1+e_2 ]_\oh - [e_1, e_1-e_2 ]_\oh
+[e_1+e_2, e_1 ]_\oh - [e_1-e_2,e_1]_\oh
\big),
\]
so $\hrank (\mB) \leq 4$. On the other hand,
the cp rank of $\mB$ is $4$
(see \cite[\S5]{Comon2008}, \cite[\S4.7]{DeSLim08}),
implying $\hrank (\mB) \geq \rank(\mB)=4$. 
Therefore, $\hrank (\mB) = 4$.
\end{example}

For an integer $r>0$, define the sets of Hermitian tensors
\begin{align}
 \label{set:Yr}
\cY_r &:=  \{\A\in \bC^{[n_1, \ldots, n_m]}:
\, \text{hrank}(\A)\le r\},  \\
\label{set:Zr}
\cZ_r &:= \{\A\in \bC^{[n_1, \ldots, n_m]}:
\, \text{hrank}(\A)=r\}.
\end{align}
Denote by $\mbox{cl}(\cY_r),\mbox{cl}(\cZ_r)$ their closures respectively,
under the Euclidean topology.
We define typical and generic Hermitian ranks as follows.

\begin{definition} \rm
An integer $r$ is called a {\em typical} Hermitian rank
of $\bC^{[n_1, \ldots, n_m]}$ if $\cZ_r$ has positive Lebsgue measure.
The smallest $r$ such that
$\mbox{cl}(\cY_r) = \bC^{[n_1, \ldots, n_m]}$
is called the {\em generic} Hermitian rank of $\bC^{[n_1, \ldots, n_m]}$,
for which we denote $r_g$.
\end{definition}

For  $m>1$ and $n_1, \ldots, n_m > 1$,
does $\bC^{[n_1, \ldots, n_m]}$ have a unique typical Hermitian rank?
If it is not unique, is the set of typical ranks consecutive?
What is the value of the generic rank $r_g$?
These questions are mostly open, to the best of the author's knowledge.
For real tensors,
we refer to~\cite{BerBlekOtt18,BlekTeit15,Blek15}
for typical and generic real tensor ranks.

For each rank-$1$ Hermitian tensor, it holds that
\[
\lmd [u_1, \ldots, u_m]_\oh = \frac{\lmd}{ |c_1|^2 \cdots |c_m|^2 }
[c_1 u_1, \ldots, c_m u_m]_\oh
\]
for all nonzero complex scalars $c_i$. That is,
$\lmd [u_1, \ldots, u_m]_\oh$ is unchanged if
we scale one entry of $u_i$ to be $1$, upon scaling $\lmd$ accordingly.
Let $\mc{W}$ be the set of Hermitian rank-$1$ tensors in
$\bC^{[n_1,\ldots,n_m]}$. Its dimension over $\re$ is
\[
\dim_{\re} \mc{W} = (2n_1-2)+\cdots + (2n_m-2) + 1
 = 2(n_1+\cdots+n_m-m)+1.
\]
The dimension of $\bC^{[n_1,\ldots,n_m]}$ over $\re$ is $(n_1\cdots n_m)^2$.
Therefore, the {\it expected Hermitian rank} of the space
$\bC^{[n_1,\ldots,n_m]}$ is
\begin{equation}
\label{formu:exprank}
\operatorname{exphrank} :=
\left\lceil
\frac{(n_1\cdots n_m)^2}{2(n_1+\cdots+n_m-m)+1}
\right \rceil.
\end{equation}
By a dimensional counting, every typical rank is always greater than or equal to
$\mbox{exphrank}$.
For the matrix case (i.e., $m=1$) and $n_1 > 2$,
the generic rank is $n_1$, which is bigger than
the expected Hermitian rank
$\left\lceil \frac{n_1^2}{2n_1-1} \right \rceil$.
Is this also true for $m>1$?
For what values of $m$ and $n_1, \ldots, n_m $,
does $\operatorname{exprank} = r_g$?
When is $\operatorname{exphrank}$ a typical rank?
These questions are mostly open.

\section{Real Hermitian Tensors}
\label{sc:real}

This section discusses real Hermitian tensors,
i.e., their entries are all real.
The subspace of real Hermitian tensors in
$\bC^{[n_1,\ldots,n_m]}$ is denoted as
\[
\re^{[n_1,\ldots,n_m]} :=
\bC^{[n_1,\ldots,n_m]} \cap
\re^{n_1 \times \cdots \times n_m \times n_1 \times \cdots \times n_m}.
\]
For real Hermitian tensors, we are interested in
their real decompositions.

\begin{definition} \rm  \label{def:real:HD}
A tensor $\cH \in \re^{[n_1,\ldots,n_m]}$ is
called {\em $\re$-Hermitian decomposable} if
\begin{equation} \label{rHD:calA}
\cH = {\sum}_{i=1}^r \lambda_i [u_i^{1}, \ldots, u_i^{m} ]_\oh
\end{equation}
for real vectors $u_i^{j} \in \bR^{n_j}$
and real scalars $\lambda_i \in \bR $.
The smallest such $r$ is called the {\em $\re$-Hermitian rank}
of $\cH$, for which we denote $\hrank_{\re}(\cH)$.
The subspace of $\re$-Hermitian decomposable tensors in
$\re^{[n_1,\ldots,n_m]}$ is denoted as
$\re_D^{[n_1,\ldots,n_m]}$.
\end{definition}

When it exists, \reff{rHD:calA} is called a $\re$-Hermitian decomposition;
if $r$ is minimum, \reff{rHD:calA} is called a $\re$-Hermitian rank decomposition.
Clearly, for all $\cH\in \re_D^{[n_1,n_2]}$,
\be \label{rHR>=cHR}
\hrank_{\re} (\cH)  \geq  \hrank (\cH).
\ee
Not every real Hermitian tensor is $\re$-Hermitian decomposable.
This is very different from the complex case.
We characterize when a tensor
is $\re$-Hermitian decomposable.

\begin{theorem}
\label{Th:real Herm} \rm
A tensor $\A \in \re^{[n_1,\ldots,n_m]}$
is $\re$-Hermitian decomposable, i.e.,
$\A \in \re_D^{[n_1,\ldots,n_m]}$, if and only if
%
%
\be
      \label{Eq:real Herm}
\A_{i_1 \ldots i_m j_1 \ldots j_m}
= \A_{k_1 \ldots k_m l_1 \ldots l_m}
\ee
for all labels such that
$\{i_s, j_s\} = \{k_s, l_s \}$, $s=1,\ldots, m$.
\end{theorem}
\begin{proof}
For convenience, denote the labeling tuples:
\[
\imath = ({i_1, \ldots, i_m, j_1, \ldots, j_m}), \,
\jmath = (k_1, \ldots, k_m, l_1, \ldots, l_m).
\]
$``\Rightarrow":$ If $\A$  has a $\re$-Hermitian decomposition
as in \reff{rHD:calA}, then
\[
\A_{\imath}=\sum_{i=1}^r \lambda_i \prod_{s=1}^m
(u_i^s)_{i_s} (u_i^s)_{j_s} =
\sum_{i=1}^r \lambda_i \prod_{s=1}^m
(u_i^s)_{k_s} (u_i^s)_{l_s} = \A_{\jmath}
\]
when $\{i_s, j_s\} = \{k_s, l_s \}$ for all $s=1,\ldots, m$.

\noindent
$``\Leftarrow":$
Assume \reff{Eq:real Herm} holds. We prove the conclusion by
induction on $m$. For $m=2$, i.e., the matrix case,
the conclusion is clearly true because every
real symmetric matrix has a real spectral decomposition.
Suppose the conclusion is true for $m$,
then we show that it is also true for $m+1$.
For $s,t \in [n_{m+1}]$, let $\mB^{s,t}$ be the tensor
in $\re^{[n_1,\ldots,n_m]}$ such that
\[
(\mB^{s,t})_{i_1 \ldots i_m j_1 \ldots j_m} =
(\A)_{i_1 \ldots i_m s j_1 \ldots j_m t}
\]
for all $i_1, \ldots, i_m j_1, \ldots, j_m$ in the range.
The condition~\reff{Eq:real Herm} implies that
$\mB^{s,t} = \mB^{t,s}$ and each $\mB^{s,t}$ is a real Hermitian tensor.
For $s < t$, define the linear map
\[
\rho_{s,t}: \, \re^{[n_1,\ldots,n_m]} \to \re^{[n_1,\ldots,n_m, n_{m+1}]},
\]
\[
[x_1,\ldots,x_m]_{\oh} \mapsto
\frac{1}{2}[x_1,\ldots,x_m,e_s+e_t]_\oh -
\frac{1}{2}[x_1,\ldots,x_m,e_s-e_t]_\oh.
\]
For $s=t$, the linear map $\rho_{s,s}$ is then defined such that
\[
\rho_{s,s} ( [x_1,\ldots,x_m]_{\oh} ) =
[x_1,\ldots,x_m,e_s]_\oh .
\]
One can verify that
$
\mA = \sum_{ 1 \leq s \leq t \leq n_{m+1} }
\rho_{s,t} ( \mB^{s,t} ).
$
By induction, each $\mB^{s,t}$ is $\re$-Hermitian decomposable,
so each $\rho_{s,t} ( \mB^{s,t} )$,
as well as $\mA$, is also $\re$-Hermitian decomposable.
\end{proof}

\begin{example} \label{exmp:Hankel:rk=4}
Consider the real Hermitian tensor $\A\in \re^{[2,2]} $ such that
\[
\A_{ijkl} = i+j+k+l
\]
for all $1\le i,j,k,l\le 2$.
It is a Hankel tensor \cite{NieYe19}.
By Theorem~\ref{Th:real Herm},
it is $\re$-Hermitian decomposable.
In fact, it has the decomposition
\[
    \A = \frac{40-13\sqrt{10}}{20} \big([u_1,e]_\oh+[e,u_1]_\oh \big) +\frac{40+13\sqrt{10}}{20} \big([u_2,e]_\oh+[e,u_2]_\oh \big),
\]
for $u_1 =(\frac{-\sqrt{10}-1}{3},1)$, $u_2 = (\frac{\sqrt{10}-1}{3},1)$.
Clearly, $\hrank_\re(\mA) \leq 4$.
Moreover, $\A$ can be expressed as the limit
\[
\A = \lim_{\eps \to 0} \eps^{-1}
\Big[ (e+\eps f)^{\otimes 4} - e^{\otimes 4}   \Big],
\]
for $f:=(1,2)$.
For this kind of tensors, the cp rank is $4$
(see \cite[\S5]{Comon2008}, \cite[\S4.7]{DeSLim08}).
Therefore, $\hrank_\re(\mA) \ge \rank (\mA) = 4$
and hence $\hrank_\re(\A) = 4$.
\end{example}

Not every basis tensor $\mc{E}^{IJ}(c)$ is $\re$-Hermitian decomposable.
For instance, the basis tensor $\mA = \mathcal{E}^{1122}(1)$
is not, because $\A_{1122}=1\neq 0=\A_{1221}$.

\begin{corollary}
For $I=(i_1, \ldots, i_m)$ and $J = (j_1,\ldots, j_m)$,
the basis tensor $\E^{IJ}(1)$
is $\re$-Hermitian decomposable if and only if
$I-J$ has at most one nonzero entry.
In particular, if $I=J$, then $\hrank_\re \E^{IJ}(1) =1$;
if $I$ and $J$ differs for only one entry,
then $\hrank_\re \E^{IJ}(1) = 2$.
\end{corollary}
\begin{proof}
The necessity direction is a direct consequence of Theorem~\ref{Th:real Herm}.
This is because if there are two distinct $k$ such that $i_k\neq j_k$,
then the condition \reff{Eq:real Herm} cannot be satisfied.
We prove the sufficiency direction
by constructing $\re$-Hermitian decompositions explicitly.
If $I = J$, then
$\E=[e_{i_1},e_{i_2},\ldots,e_{i_m}]_\oh$
and $\hrank_\re \E^{IJ}(1) =1$.
If $I$ and $J$ differs for only one entry, say, $i_k\neq j_k$, then
\[
\E=\frac{1}{2} [e_{i_1},e_{i_2},\ldots,e_{i_k}+e_{j_k},
\cdots e_{i_m}]_\oh - \frac{1}{2}
[e_{i_1},e_{i_2},\ldots,e_{i_k}-e_{j_k},\cdots e_{i_m}]_\oh
\]
and hence $\hrank_\re \E^{IJ}(1) \leq 2$.
Since $\hrank_\re \E^{IJ}(1) \ge \hrank \E^{IJ}(1)= 2$,
we must have $\hrank_\re \E^{IJ}(1) = 2$.
\end{proof}

The major reason for not all real Hermitian tensors are $\re$-Hermitian decomposable
is because of the dimensional difference. That is, the dimension of
$\re_D^{[n_1,\ldots,n_m]}$ is less than that of $\re^{[n_1,\ldots,n_m]}$.
By Theorem~\ref{Th:real Herm}, the dimension of
$\re_D^{[n_1,\ldots,n_m]} $ is equal to the cardinality of the set
$
\{(i_1,\ldots,i_m,j_1,\ldots,j_m) :\, 1 \le i_k \le  j_k \le n_k\}.
$
Thus
\begin{equation}
\dim \re_D^{[n_1,\ldots,n_m]} =
\prod_{k=1}^m \binom{n_k+1}{2} =
\prod_{k=1}^m \frac{ n_k (n_k+1) }{2}.
\end{equation}
However, the dimension of $\re^{[n_1,\ldots,n_m]}$ is
\be
\dim \re^{[n_1,\ldots,n_m]} =
\binom{ N + 1 }{2} , \quad N = n_1 \cdots n_m.
\ee
The dimension of $\re^{[n_1,\ldots,n_m]}$
equals the dimension of $\mc{S}^N$,
the space of $N$-by-$N$ real symmetric matrices.
The dimension of $\re_D^{[n_1,\ldots,n_m]}$
equals the dimension of the tensor product space
$
\mc{S}^{n_1} \otimes \cdots \otimes \mc{S}^{n_m}.
$
If $m>1$ and all $n_i > 1$, then
\be \label{dimRD<dimR}
\dim \re_D^{[n_1,\ldots,n_m]} \, < \,  \dim \re^{[n_1,\ldots,n_m]} .
\ee

Real Hermitian decompositions can also be equivalently expressed
in terms of real polynomials.
Let each $x_i \in \re^{n_i}$ be a real vector variable.
The real decomposition~\reff{rHD:calA} implies that
\be \label{real:H(x,x)=sum:qi}
\cH(x,x) = {\sum}_{i=1}^r \lmd_i
\big( (u_i^1)^Tx_1 \big)^2 \cdots \big( (u_i^m)^Tx_m \big)^2  .
\ee
When $\cH$ is $\re$-Hermitian decomposable,
\reff{real:H(x,x)=sum:qi} also implies \reff{rHD:calA}.

\begin{lemma} \label{rH:poly=decomp}
For real vectors $u_i^j$, a tensor $\cH\in \re_D^{[n_1,\ldots,n_m]}$
has the decomposition \reff{real:H(x,x)=sum:qi}
if and only if the $\re$-Hermitian decomposition \reff{rHD:calA} holds.
\end{lemma}
\begin{proof}
The ``if" direction is obvious. We prove the ``only if" direction.
Let $\mc{U} = {\sum}_{i=1}^r \lmd_i [u_i^1, \ldots, u_i^m]_\oh$.
Then
$
\langle \cH - \mc{U}, [x_1, \ldots, x_m]_\oh \rangle = 0
$
for all real $x_i \in \re^{n_i}$. Since
$\cH - \mc{U} \in \re_D^{[n_1,\ldots,n_m]}$,
$\langle \cH - \mc{U}, \cH - \mc{U} \rangle = 0$,
so $\cH = \mc{U}$ and \reff{rHD:calA} holds.
\end{proof}

In the following, we study the relationship between
real and complex Hermitian decompositions.

\begin{lemma} \label{coro:complex to real}
Suppose $\cH\in \re_D^{[n_1,\ldots,n_m]}$ has the decomposition
\[
\cH = {\sum}_{j=1}^r \lambda_j[u^1_{j},u^2_j,\ldots,u^m_j]_\oh ,
\]
with complex $u_j^i \in \bC^{n_i}, 0\neq \lambda_j \in \re$. Let
\[
U :=\bbm
(u^1_1\boxtimes \overline{u^1_1} \boxtimes \cdots   u^{m-1}_1 \boxtimes \overline{u^{m-1}_1 }),
& \cdots, &
(u^1_r\boxtimes \overline{u^1_r} \boxtimes \cdots u^{m-1}_r \boxtimes \overline{u^{m-1}_r} )
\ebm.
\]
If $k :=\rank(U)\in  \{1,2,r\}$, then
\begin{equation}  \label{eq:real}
 \cH = {\sum}_{j=1}^{r}\beta_j [u^1_{j},u^2_j,\ldots, u^{m-1}_j, s^m_j]_\oh
\end{equation}
for real vectors $s^m_j\in \re^{n_m}$ and real scalars $\beta_j \in \re $.
\end{lemma}
\begin{proof}
Let $\kappa_\phi$ be the canonical Kronecker flattening map in \reff{flat:Kron:m=2}, then
\[
H :=\kappa_\phi(\cH)  =
{\sum}_{j=1}^r \lambda_j U_j(u^m_{j} \boxtimes \overline{u^m_j})^T
= {\sum}_{j=1}^r \lambda_j  U_j( \overline{u^m_j}\boxtimes u^m_j)^T,
\]
where $U_j$ denotes the $j$th column of $U$.
The second equality holds, since $\cH$ is $\re$-Hermitian decomposable. Thus,
$
\sum_{j=1}^r \lambda_j  U_j(u^m_{j} \boxtimes
\overline{u^m_j}-\overline{u^m_j}\boxtimes u^m_j)^T =0.
$
\begin{itemize}

\item If $k=r$, then $\{U_1,\ldots,U_r\}$
is linearly independent, which implies
$u^m_{j} \boxtimes \overline{u^m_j}-\overline{u^m_j}\boxtimes u^m_j = 0$
for all $j$. So $u^m_{j} \boxtimes \overline{u^m_j} $ is real.
There exists $s^m_j\in \re^{n_m} $
such that $u^m_{j} \boxtimes \overline{u^m_j}  = s^m_j\boxtimes s^m_j $.
It gives a desired decomposition as in \reff{eq:real}.

\item If $k=1$, then there exists $\alpha_j  \in \re$ such that
$U_j = \alpha_j U_1$ for $1\le j \le r$. Thus
\[
H = U_1 V_1^T  = U_1 \overline{V_1}^T
\quad \mbox{where} \quad
V_1:= {\sum}_{j=1}^r \alpha_j \lambda_j u^m_{j} \boxtimes \overline{u^m_j} .
\]
Since $U_1 (V_1-\overline{V_1})^T = 0$,
$V_1$ is the vectorization of a real symmetric matrix,
then there exist $s^m_j\in \re^{n_m} $ and $\beta_j \in \re $ such that
$
V_1 = {\sum}_{j=1}^r \beta_j s^m_{j} \boxtimes s^m_j .
$
It also gives a desired decomposition as in \reff{eq:real}.

\item If $k=2$, we can generally assume that
$U_1,U_p$ are linearly independent.
For each $i \not\in \{1,p\}$, $U_i$ is a linear combination of
$U_1,U_p$. Since each $U_i$
is the vectorization of a rank-$1$ Hermitian matrix,
$U_i$ must be a multiple of $U_1$ or $U_p$, say,
$U_i = U_1$ for  $1\le i \le p-1$
and $U_i = U_p$ for $p \le i\le r$,
up to scaling of $\lmd_i$. Thus,
\[
H =   U_1 X_1^T + U_p X_2^T
=  U_1 \overline{X_1}^T + U_p \overline{X_2}^T,
\]
where $X_1 :=\sum_{i=1}^{p-1}\lambda_i u^m_{i} \boxtimes \overline{u^m_i}$,
$X_2 :=\sum_{j=p}^{r}\lambda_j u^m_{j} \boxtimes \overline{u^m_j}$.
Since  $U_1 (X_1 - \overline{X_1})^T + U_p (X_2-\overline{X_2})^T =0$,
we have $X_1 = \overline{X_1}$ and $X_2 =\overline{X_2}$,
so $X_1, X_2$ are vectorizations of real symmetric matrices.
There exist $s^m_j \in \re^{n_m},\beta_j \in \re $ such that
$
X_1
=  {\sum}_{i=1}^{p-1}\beta_i s^m_{i} \boxtimes \overline{s^m_i}
$,
$
X_2
=  {\sum}_{j=p}^{r}\beta_j s^m_{j} \boxtimes \overline{s^m_j}.
$
This also gives a desired decomposition as in \reff{eq:real}.
\end{itemize}
For every case of $k=1,2,r$, we get a decomposition like \reff{eq:real}.
\end{proof}

Based on the above lemma, we can get the following conclusion.

\begin{prop} \label{pro:rHD:r=c}
For $\cH \in \re_D^{[n_1,\ldots,n_m]}$,
if $\hrank(\cH) \le 3$, then $\hrank(\cH) = \hrank_\re(\cH)$.
Furthermore, if $\hrank_\re(\cH) \le 4$, then $\hrank(\cH) = \hrank_\re(\cH)$.
\end{prop}
\begin{proof}
Let $r:= \hrank(\cH)$.
We consider $r>0$ (the case $r=0$ is trivial).
If $r \le 3$, we can apply
Lemma~\ref{coro:complex to real} to $\cH$.
Note that $k := \rank U \in \{1,2,r\}$,
since $r\le 3$. For each $i=1,\ldots,m $, the set $\{u_j^i \}_{j=1}^r $
can be changed to a set of real vectors
 while the length of decomposition does not change.
As a result, we get a  $\re$-Hermitian decomposition
for $\cH$ with length $r$, so
$\hrank_\re(\cH) = \hrank(\cH)$.

If $\hrank_\re(\cH) \le 4 $, then $\hrank(\cH)\le 4$.
If $\hrank(\cH) \le 3 $, then the previous argument proves
$\hrank(\cH) = \hrank_\re(\cH)$. If $\hrank(\cH)=4 $,
then $\hrank_\re(\cH)\ge 4$, and hence $\hrank_\re(\cH)=\hrank(\cH)=4$.
\end{proof}

For $\re$-Hermitian decomposable tensors,
the concepts of border generic, typical and expected ranks
can be similarly defined, as in the Subsection~\ref{ssc:hbr}.
The discussion is the same as for the complex case.
We omit this for cleanness of the paper.

\section{Matrix flattenings}
\label{sc:MatFlat}

All classical matrix flattenings are applicable to Hermitian tensors.
In particular, Hermitian and Kronecker flattenings are special for Hermitian tensors.

\subsection{Hermitian flattening}
\label{ssc:Hflat}

Define the linear map
$\mathfrak{m}: \cpx^{[n_1, \ldots,n_m]}  \to  \mathbb{M}^{N}$
($N =n_1 \cdots n_m$) such that for all $v_i\in \bC^{n_i}$, $i=1,\ldots,m$,
\be  \label{df:map:frakm}
\baray{c}
\mathfrak{m} \big( [v_1, v_2, \ldots,  v_m]_\oh  \big) \, = \,
(v_1 v_1^*) \boxtimes (v_2 v_2^*) \boxtimes \cdots \boxtimes (v_m v_m^*) ,
\earay
\ee
where $\boxtimes$ denotes the classical Kronecker product.
The map $\mathfrak{m}$ is a bijection between
$\bC^{[n_1,\ldots,n_m]}$ and
$
\mathbb{M}^{N}  \cong  \mathbb{M}^{n_1} \otimes \cdots \otimes \mathbb{M}^{n_m}.
$
The Hermitian decomposition $\cH = \sum_{i=1}^r \lmd_i
[u_i^1, \ldots, u_i^m]_\oh$ is equivalent to that
\be  \label{H=qq*:rk1}
\left\{ \baray{rcl}
\mathfrak{m}(\cH)
& = & {\sum}_{i=1}^r \lmd_i \,
( u_i^1(u_i^1)^* )\boxtimes  \cdots \boxtimes (u_i^m (u_i^m)^*) \\
& = & {\sum}_{i=1}^r \lmd_i \,
 ( u_i^1 \boxtimes \cdots \boxtimes u_i^m )
 ( u_i^1 \boxtimes \cdots \boxtimes u_i^m )^* .
\earay \right.
\ee
The matrix $H := \mathfrak{m}( \cH )$
is called the {\em Hermitian flattening matrix} of $\cH$.
It can be labelled by
$I=(i_1,\ldots, i_m)$ and $J=(j_1,\ldots, j_m)$ such that
\be \label{HIJ=Hij}
(H)_{IJ} = \cH_{i_1 \ldots i_m j_1 \ldots j_m} .
\ee
The following is a basic result about flattening and ranks.

\begin{lemma} \label{lema:matrix rank}
If $H = \mathfrak{m}(\cH)$, then
$\hrank (\mathcal{H}) \ge  \hbrank (\mathcal{H}) \ge \rank (H)$.
\end{lemma}
\begin{proof}
The first inequality is obvious. We prove the second one.
Let $r := \hbrank(\mathcal{H})$,
then there is a sequence $\{ \mH_k \} \subseteq \bC^{[n_1,\ldots, n_m]}$
such that $\mH_k \to \mH$ and $\hrank \mH_k = r$.
Let $H_k := \mathfrak{m}( \mH_k)$, then $H_k \to H$ and $\rank\, H_k \leq r$,
so $\rank \,(H) \leq r$.
\end{proof}

It is possible that $\text{hrank}(\mathcal{H}) > \text{rank}(H)$.
For instance, consider the basis tensor
$\mc{E}^{(11)(22)}(1)$. Its Hermitian flattening matrix has rank $2$
while the Hermitian rank is $4$ (see Example~\ref{Ex:Herm Decom m=2}).

For each $\cH \in \re_D^{[2,2]}$, its Hermitian flattening matrix is in the form
\be \label{redc:H=ACB}
 \mathfrak{m}(\cH)  =
\begin{pmatrix}
  A & C \\  C & B
 \end{pmatrix}, \quad \mbox{where} \quad
  A, B, C \in  \mc{S}^2.
\ee

\begin{prop} \label{thm:normal form}
For each $\cH \in \re_D^{[2,2]}$ as above,
there exist invertible matrices $P,Q\in \re^{2\times 2}$ such that $\tilde{\cH}:=(P,Q)\times_{cong} \cH $ has the flattening
\be  \label{eq: normal form}
\mathfrak{m}(\tilde{\cH}) =
\bpm    sI_2 & D \\ D & s\tilde{B} \epm
 - s
\bpm  uu^T & 0 \\   0 & 0 \epm,
\ee
where $s\in \{0,1,-1\}$, $D$ is real diagonal,
$u\in \re^{2} $ and $\tilde{B} \in \mc{S}^2$ .
In particular, $u=0$ if one of $A,B$ is positive (or negative) definite,
and $s=0$ if $A=B=0$.
\end{prop}
\begin{proof}
Case I: Assume one of $A,B$ is nonzero, say, $A \ne 0$.
%
%
If $A$ is not negative semidefinite, there is $v\in \re^2$ such that $A+vv^T \succ 0$.
Then there is $U \in \re^{2 \times 2}$ such that $U(A+vv^T)U^T = I_2$.
There exists an orthogonal matrix $V$ such that $D:=V(UCU^T)V^T$ is diagonal.
Let $\tilde{\cH}:=(I_2,VU)\times_{cong} \cH $, then
\[
\begin{array}{ll}
            \mathfrak{m}(\tilde{\cH}) &=
            \begin{pmatrix}
                V(U(A+vv^T)U^T)V^T & V(UCU^T)V^T\\
                V(UCU^T)V^T & V(UBU^T)V^T
            \end{pmatrix}
            -
            \begin{pmatrix}
                V(Uvv^TU^T)V^T & 0\\
                0& 0
            \end{pmatrix} \\
            &=
            \begin{pmatrix}
                sI_2 & D \\
                D & s\tilde{B}
            \end{pmatrix}
            -
            s\begin{pmatrix}
                uu^T & 0\\
                0 & 0
            \end{pmatrix} .
 \end{array}
\]
So, the decomposition~\ref{eq: normal form} holds for
$s=1, \tilde{B}:= V(UBU^T)V^T, u:=VUv$. If $A$ is negative semidefinite,
then $-A$ is not negative semidefinite.
We do the same thing for $-\cH$ and can get \ref{eq: normal form} with $s=-1$.
In particular, if either $A$ or $B$ is positive (or negative) definite,
we can choose $v=0$ and thus $u=VUv=0$.

Case II: Assume $A=B=0$. Since $C$ is real symmetric, there exists a matrix $U$ such that $D:=UCU^T$ is diagonal. Let $\tilde{\cH}:=(I_2,U)\times_{cong} \cH $, then
\[
\mathfrak{\tilde{\cH}} =
\begin{pmatrix}
   0 & UCU^T \\
  UCU^T & 0
\end{pmatrix} =
\begin{pmatrix}
    0 & D \\
    D & 0
\end{pmatrix}.
 \]
For this case, $s=0$.
\end{proof}

Suppose the diagonal matrix $D$ in \reff{eq: normal form} is
$D = \mbox{diag}(d_1, d_2)$.
When $s=0$, the tensor $\tilde{\cH} $ has the Hermitian decomposition:
\[
\frac{1}{2}d_1\left( [\begin{pmatrix}
            1 \\ 1
        \end{pmatrix},
        \begin{pmatrix}
            1 \\ 0
        \end{pmatrix}
        ]_\oh -  [\begin{pmatrix}
            1 \\ -1
        \end{pmatrix},
        \begin{pmatrix}
            1 \\ 0
        \end{pmatrix}
        ]_\oh \right)+
        \frac{1}{2}d_2\left( [\begin{pmatrix}
            1 \\ 1
        \end{pmatrix},
        \begin{pmatrix}
            0 \\ 1
        \end{pmatrix}
        ]_\oh -  [\begin{pmatrix}
            1 \\ -1
        \end{pmatrix},
        \begin{pmatrix}
            0 \\ 1
        \end{pmatrix}
        ]_\oh \right) .
\]
Thus, $\hrank_\re(\tilde{\cH})\le 4 $. When $s=1$ or $-1$, let
$E:=s\tilde{B}-s \cdot\text{diag}(d_1^2,d_2^2) $. Suppose
$E = \lambda_1 v_1v_1^T+\lambda_2 v_2v_2^T $ is an orthogonal eigenvalue decomposition.
Then, (note that $s^2=1$),
\begin{multline*}
\mathfrak{m}(\tilde{\cH}) = s
        \begin{pmatrix}
            1 & sd_1 \\ sd_1 & s^2d_1^2
        \end{pmatrix} \boxtimes
        \begin{pmatrix}
            1 & 0 \\ 0 & 0
        \end{pmatrix}
        +
        s
        \begin{pmatrix}
            1 & sd_2 \\ sd_2 & s^2d_2^2
        \end{pmatrix} \boxtimes
        \begin{pmatrix}
            0 & 0 \\ 0 & 1
        \end{pmatrix}
        +  \\
        \begin{pmatrix}
            0 & 0 \\ 0 & 1
        \end{pmatrix} \boxtimes (\lambda_1 v_1v_1^T+\lambda_2 v_2v_2^T)
        -
        s\begin{pmatrix}
            1 & 0 \\ 0 & 0
        \end{pmatrix}\boxtimes (uu^T) .
\end{multline*}
The above gives the real Hermitian decomposition for $\tilde{\cH} $:
\begin{multline*}
        \tilde{\cH} = s[\begin{pmatrix}
            1 \\ sd_1
        \end{pmatrix},
        \begin{pmatrix}
            1 \\ 0
        \end{pmatrix}
        ]_\oh
        +
        s[\begin{pmatrix}
            1 \\ sd_2
        \end{pmatrix},
        \begin{pmatrix}
            0 \\ 1
        \end{pmatrix}
        ]_\oh
        +\lambda_1
        [\begin{pmatrix}
            0 \\ 1
        \end{pmatrix},
        v_1
        ]_\oh
        + \\
        \lambda_2
        [\begin{pmatrix}
            0 \\ 1
        \end{pmatrix},
        v_2
        ]_\oh
        -s
        [\begin{pmatrix}
            1 \\ 0
        \end{pmatrix},
        u
        ]_\oh .
\end{multline*}
For all cases, we have $\hrank_\re(\tilde{\cH})\le 5 $. Since
$\tilde{\cH} = (P,Q)\times_{cong}\cH $ and $P,Q$ are invertible,
$\hrank_\re(\cH)= \hrank_\re(\tilde{\cH})$.
Therefore, we get the following conclusion.

\begin{theorem} \label{thm:RD22:rank}
For every $\cH\in \re_D^{[2,2]}$, with the flattening as in \reff{redc:H=ACB},
we have $\hrank_\re(\cH) \le 5$.
In particular, we have $\hrank(\cH) = \hrank_\re(\cH) \le 4$ if one of $A,B$
is positive (or negative) definite, or if $A=B=0$.
\end{theorem}
\begin{proof}
The inequality $\hrank_\re(\cH) \le 5$ is implied by $ \hrank_\re(\tilde{\cH})\le 5$
and $\hrank_\re(\cH)= \hrank_\re(\tilde{\cH})$. If one of $A,B$
is positive (or negative) definite, then $u=0$ by Proposition~\ref{thm:normal form}
and hence $\hrank_\re(\cH)= \hrank_\re(\tilde{\cH}) \le 4 $.
If $A=B=0$, we already have $\hrank_\re(\cH)= \hrank_\re(\tilde{\cH}) \le 4 $. By Proposition \ref{pro:rHD:r=c}, $\hrank(\cH) = \hrank_\re(\cH) \le 4$ if one of $A,B$
is positive (or negative) definite, or if $A=B=0$.
\end{proof}

For a general $\cH \in \re^{[2,2]} $, its Hermitian flattening matrix is
\[
\mathfrak{m} (\cH) \, = \,
\begin{pmatrix}
  A & C \\
  C^T & B
\end{pmatrix}
\]
where $A,B$ are real symmetric and $C$ is generally not symmetric.
By doing the same thing as above, we can congruently transform
$\cH$ to $\tilde{\cH}$ such that
\[
\mathfrak{m}(\tilde{\cH})  =      \begin{pmatrix}
            sI_2 & \tilde{C} \\
            \tilde{C}^T & sD
        \end{pmatrix}
        -
        s\begin{pmatrix}
            uu^T & 0 \\
            0 & 0
        \end{pmatrix},
\]
where $s \in \{0, 1, -1\}$, $D$ is diagonal but $\tilde{C}$
is still generally not symmetric.
However, the above does not produce Hermitian decompositions
with desired lengths as for the case \reff{eq:  normal form}.

Hermitian ranks can be investigated through the Hermitian flattening.
For $A \in \mathbb{M}^{N}$,
we define its $\mathbb{M}$-rank as
\be \label{M-rank:A}
\rank_{\mathbb{M}} A \, := \, \min \left\{ r
\left|\baray{c}
A = \sum_{i=1}^r
\lmd_i (a_i^1(a_i^1)^*) \boxtimes \cdots (a_i^m(a_i^m)^*), \\
\lmd_i \in \re,\quad a_i^j \in \cpx^{n_j}.
\earay\right.
\right\}.
\ee
Generic and typical $\mathbb{M}$-ranks
can be similarly defined for $\mathbb{M}^{N}$ as in \S\ref{ssc:hbr}.

\begin{theorem}
For every $\mH \in \bC^{[n_1, \ldots, n_m]}$, we have
$\hrank \mH = \rank_{\mathbb{M}} \mathfrak{m}(\mH)$.
Moreover, an integer $r$ is the generic (resp., a typical) rank for
$\bC^{[n_1, \ldots, n_m]}$ if and only if $r$ is
is the generic (resp., a typical) rank for $\mathbb{M}^{N}$.
\end{theorem}
\begin{proof}
The equality $\hrank \mH = \rank_{\mathbb{M}} \mathfrak{m}(\mH)$
follows from the equation \reff{H=qq*:rk1}.
Since $\mathfrak{m}$ is a bijection between $\bC^{[n_1, \ldots, n_m]}$ and
$\mathbb{M}^{N}$, an integer $r$ is the generic (resp., a typical) rank for
$\bC^{[n_1, \ldots, n_m]}$ if and only if $r$
is the generic (resp., a typical) rank for $\mathbb{M}^{N}$.
\end{proof}

\subsection{Kronecker flattening}
\label{ssc:kron}

Every matrix flattening map $\phi$ on the tensor space
$\cpx^{n_1 \times \ldots \times n_m}$
can be used to define a new flattening map $\kappa_\phi$ on
$\bC^{[n_1,\ldots,n_m]}$.
Suppose $\phi$ flattens tensors in $\cpx^{n_1 \times \ldots \times n_m}$
to matrices of the size $D_1$-by-$D_2$.
Then we can define the linear map $\kappa_\phi:
\bC^{[n_1,\ldots,n_m]}\to \bC^{D_1^2 \times D_2^2 }$ such that
\be  \label{kronflat:kapphi}
\kappa_\phi \big( [u_1, \ldots, u_m]_{\oh} \big)  =
\phi(u_1 \otimes \cdots \otimes u_m) \boxtimes
 \overline{ \phi(u_1 \otimes \cdots \otimes u_m) }
\ee
for all $u_i\in \bC^{n_i} $.
The map $\kappa_\phi$ is called the {\em $\phi$-Kronecker} flattening
generated by $\phi$. When $\phi$ is the standard flattening
such that $\phi(a_1 \otimes \cdots a_{m-1} \otimes a_m)
= (a_1 \boxtimes \cdots a_{m-1}) (a_m)^T$, then
$\kappa_\phi$ is the linear map such that
\be \label{flat:Kron:m=2}
\kappa_\phi \Big( {\sum}_i \lmd_i [u^1_i, \ldots, u^m_i]_{\oh} \Big)
= {\sum}_i \lmd_i
Z_i \boxtimes \overline{ Z_i }
\ee
where $Z_i := (u_i^1 \boxtimes \cdots \boxtimes u_i^{m-1}) (u_i^m)^T$.
The map $\kappa_\phi$ in \reff{flat:Kron:m=2}
is called the {\em canonical Kronecker flattening}.

\begin{lemma} \label{lm:kron:flat}
Let $\phi$ be a flattening map on $\cpx^{n_1 \times \cdots n_m}$
and $\kappa_\phi$ be the corresponding $\phi$-Kronecker flattening.
Then, for each $\cH \in \cpx^{[n_1, \ldots, n_m]}$,
\be  \label{kronflat:rkrl}
\hrank(\mathcal{H}) \ge  \hbrank(\mathcal{H})
\ge \rank \kappa_\phi (\cH).
\ee
\end{lemma} 

The above is an analogue of Lemma~\ref{lema:matrix rank}.
We omit its proof for cleanness of the paper.
The Hermitian and Kronecker flattening
may give different lower bounds for Hermitian ranks, as shown below.

\begin{example}\label{prop:rank 1 flattening}
For $m=2$ and $n>1$, consider the Hermitian tensor in $\re^{[n,n]}$
\[
\cH = {\sum}_{i,j=1}^n e_i\otimes e_i \otimes  e_j \otimes e_j
= \big( {\sum}_{i=1}^n e_i\otimes e_i \big) \otimes \big({\sum}_{i=1}^n e_i\otimes e_i \big).
\]
Let $\kappa_\phi$ be the canonical Kronecker flattening as in
\reff{flat:Kron:m=2}, then
\[
\mathfrak{m}(\cH) = \big({\sum}_{i=1}^n e_i \boxtimes e_i \big)
\big({\sum}_{i=1}^n e_i \boxtimes e_i\big)^T, \,
\kappa_\phi (\cH ) =
\big({\sum}_{i=1}^n e_i e_i^T \big) \boxtimes
\big({\sum}_{i=1}^n e_i e_i^T \big) = I_{n^2}.
\]
By Lemma~\ref{lm:kron:flat},
$\hrank (\cH) \geq \rank \kappa_\phi (\cH ) =  n^2$
while $\rank \, \mathfrak{m}(\cH) = 1$ .
Indeed, we further have a sharper lower bound
\[
\hrank(\cH)   \ge  n^2+1.
\]
Suppose otherwise that $\hrank(\cH) = n^2$, say,
$
\cH = {\sum}_{i=1}^{n^2}\lambda_i[u_i,v_i]_\oh
$
for $\lmd_i \in \re$ and $u_i,v_i \in \cpx^n$, then
\[
\kappa_\phi (\cH ) =I_{n^2}
= {\sum}_{i=1}^{n^2} \lambda_i(u_i\cdot v_i^T) \boxtimes ( \overline{u_i}\cdot \overline{v_i}^T )
= {\sum}_{i=1}^{n^2} \lambda_i(u_i \boxtimes \overline{u_i}) (v_i \boxtimes \overline{v_i})^T .
\]
Let
\[
U=[\lambda_1u_1 \boxtimes \overline{u_1},\ldots, \lambda_{n^2} u_{n^2} \boxtimes \overline{u_{n^2}}],
\quad
V=[v_1 \boxtimes \overline{v_1},\ldots, v_{n^2} \boxtimes \overline{v_{n^2}}].
\]
Then $U,V$ are square matrices of length $n^2$ and
\[
 UV^T=I_{n^2} \Rightarrow V^TU=I_{n^2} \Rightarrow
 \lambda_j(v_i \boxtimes \overline{v_i})^T(u_j \boxtimes \overline{u_j})=\left \{
 \begin{array}{ll}
  1 & i=j, \\
  0 & i\neq j .
 \end{array} \right.
\]
For $i\neq j$, we have
\[
(v_i \boxtimes \overline{v_i})^T(u_j \boxtimes \overline{u_j})
=(v_i^Tu_j)\boxtimes (\overline{v_i}^T \overline{u_j})=|v_i^Tu_j|^2=0 \Rightarrow v_i^Tu_j=0.
\]
Thus, $u_2,\ldots,u_{n^2} \in v_1^\perp$ and
\[
r := \dim( \mbox{span} \{u_2,\ldots,u_{n^2}\}) \le n-1.
\]
Let $\{s_1,\ldots,s_r\}$ be a basis for
$\mbox{span}\{u_2,\ldots,u_{n^2}\}$.
For each $i=2,3,\ldots n^2$, $u_i \boxtimes \overline{u}_i$
belongs to the span of the set $\{s_p\boxtimes \overline{s}_q\}_{1\le p,q\le r}$, so
\[
\dim \Big(\mbox{span} \{u_i \boxtimes \overline{u}_i\}_{i=2}^{n^2} \Big) \le
\dim \Big(\mbox{span} \{s_p\boxtimes \overline{s}_q\}_{1\le p,q\le r} \Big) =r^2.
\]
This implies that
\[
n^2 = \rank(U) \le 1 + \dim \Big(
\mbox{span} \{u_i \boxtimes \overline{u}_i\}_{i=2}^{n^2} \Big)
\le r^2+1 \le (n-1)^2+1.
\]
However, $n^2>(n-1)^2+1$ when $n\ge 2$.
This is a contradiction, so $\hrank(\cH)\ge n^2+1$.
For the case $n=2$, $\hrank(\cH) = n^2+1$,
because we have a Hermitian decomposition of length $5$
(in the following $c := \sqrt{1+\sqrt{2}}$):
\begin{multline*}
\frac{1}{2c^4-2} \Bigg(
    \left[
    \begin{pmatrix}
        c \\ 1
    \end{pmatrix},
    \begin{pmatrix}
        c \\ 1
    \end{pmatrix} \right ]_\oh
    +
    \left[
    \begin{pmatrix}
        c \\ -1
    \end{pmatrix},
    \begin{pmatrix}
        c \\ -1
    \end{pmatrix} \right ]_\oh
    -
    \left[
    \begin{pmatrix}
        1 \\ c \sqrt{-1}
    \end{pmatrix},
    \begin{pmatrix}
        1 \\ c \sqrt{-1}
    \end{pmatrix} \right ]_\oh  \\
-  \left[
    \begin{pmatrix}
        1 \\ -c \sqrt{-1}
    \end{pmatrix},
    \begin{pmatrix}
        1 \\ -c \sqrt{-1}
    \end{pmatrix} \right ]_\oh \Bigg)
+  2\left[
    \begin{pmatrix}
        0 \\ 1
    \end{pmatrix},
    \begin{pmatrix}
        0 \\ 1
    \end{pmatrix} \right ]_\oh.
\end{multline*}
When $n>2$, the true value of $\hrank(\cH)$ is not known to the authors.
\end{example}

\subsection{Orthogonal decompositions}
\label{ssc:spec}

For each $\mathcal{U} \in \mathbb{C}^{n_1\times\cdots\times n_m}$,
the conjugate tensor product $\mathcal{U}\otimes \overline{ \mc{U} }$
is always Hermitian. 
In fact, each Hermitian tensor can be written as a sum of
such conjugate tensor products \cite{Ni19}.
For each $\mH \in \bC^{[n_1, \ldots, n_m]}$,
its Hermitian flattening matrix $H = \mathfrak{m}(\mH)$ is Hermitian.
Let $s := \rank H$ and suppose $H$ has the spectral decomposition
\[
H = \lmd_1 q_1 q_1^* + \cdots + \lmd_s q_s q_s^*,
\]
where $\lmd_i$'s are the real eigenvalues
and $q_1,\ldots, q_s$ are the orthonormal eigenvectors in $\cpx^N$.
Let $\mc{U}_i$ be the tensor in $\bC^{n_1\times\cdots\times n_m}$
such that $q_i = \mbox{vec}( \mc{U}_i )$, then
\be  \label{Eq:tensormatrixdecom}
\cH = {\sum}_{i=1}^s \lambda_i
\mathcal{U}_i \otimes \overline{\mathcal{U}_i}.
\ee
Note each $\| \mc{U}_i \|= \|q_i\| = 1$ and
$\langle \mathcal{U}_i, \mathcal{U}_j \rangle = q_j^*q_i = 0$ for $i \ne j$.
%
%

In \reff{Eq:tensormatrixdecom}, if each $\mc{U}_i$ is a rank-$1$ tensor, then
it gives an orthogonal Hermitian decomposition.
As in \cite{Ni19}, $\cH$ is called {\it unitarily Hermitian decomposable}
if $\mH = \sum_{i=1}^r \lmd_i [u_i^1, \ldots, u_i^m]_\oh$
for real scalars $\lmd_i$ and unit length vectors $u_i^j$ such that
\be
\big( \big(u_i^1 \big)^*u_j^1 \big)   \cdots
\big( \big(u_i^m\big)^*u_j^m \big)   = 0  \quad ( i \ne j ).
\ee
If all $u_i^j$ are real, then such $\cH$
is called {\it orthogonally Hermitian decomposable}.
For convenience, $\cH$ is said to be
$\mathbb{U}$-Hermitian (resp., $\mathbb{O}$-Hermitian) decomposable
if it is unitarily (resp., orthogonally) Hermitian decomposable.
The detection of $\mathbb{U}$/$\mathbb{O}$-Hermitian decomposability
can be done by checking its Hermitian flattening matrix.
Note that
$\mH = \sum_{i=1}^r \lmd_i [u_i^1, \ldots, u_i^m]_\oh$
if and only if
\[
\mathfrak{m}(\mH) = {\sum}_{i=1}^r \lmd_i  (u_i^1 \boxtimes \cdots \boxtimes u_i^m)
(u_i^1 \boxtimes \cdots \boxtimes u_i^m)^*.
\]
When $\mH$ is $\mathbb{U}$/$\mathbb{O}$-Hermitian decomposable,
the above gives a spectral decomposition for $H$.
When nonzero eigenvalues of $H$ are distinct from each other,
its spectral decomposition is unique.
For such cases, $\cH$ is  $\mathbb{U}$/$\mathbb{O}$-Hermitian decomposable
if and only if each $\rank \mathcal{U}_i = 1$.
When $H$ has a repeated nonzero eigenvalue,
deciding $\mathbb{U}$/$\mathbb{O}$-Hermitian decomposability becomes harder.
We refer to \cite{Ni19,QiLuo17} for more about tensor eigenvalues.

\section{PSD Hermitian tensors}
\label{sc:nhe}

A Hermitian tensor $\cH$ is uniquely determined
by the multi-quadratic conjugate polynomial
$
\cH(x,\overline{x}) := \langle \cH,  [x_1, \ldots, x_m]_{\otimes_h}\rangle,
$
in the tuple $x := (x_1, \ldots, x_m)$ of complex vector variables
$x_i \in \cpx^{n_i}$.
Like the matrix case, positive semidefinite Hermitian tensors
can be naturally defined \cite{Ni19}.

\begin{definition} \label{def:nng}
Let $\F = \cpx$ or $\re$.
A Hermitian tensor $\cH \in \F^{[n_1,\ldots,n_m]}$
is called {\it $\F$-positive semidefinite} ($\F$-psd)
if $\cH(x,\overline{x}) \geq 0$ for all $x_i \in \F^{n_i}$.
Moreover, if $\cH(x,\overline{x}) > 0$ for all $0 \ne x_i \in \F^{n_i}$,
then $\cH$ is called {\it $\F$-positive definite} ($\F$-pd).
\end{definition}

For convenience, a complex (resp., real) Hermitian tensor
is called psd if it is $\cpx$-psd (resp., $\re$-psd).
Denote the cone of $\F$-psd Hermitian tensors
\be \label{psdcone:CR}
\baray{rcl}
\mathscr{P}_{\F}^{[n_1,\ldots,n_m]} & := &
\left\{\cH \in \F^{[n_1,\ldots,n_m]}: \,
\cH(x,\overline{x}) \geq 0 \, \forall \, x_i \in \F^{n_i}
\right \}.
\earay
\ee

\begin{example} \label{ex:CR-psd}
(i) Consider $\cH \in \bC^{[3,3]}$ such that
$\cH(x,y) = \langle \cH, [x, y]_\oh \rangle$ is the following conjugate polynomial
(for cleanness of display, the variable $x_1$ is changed to $x :=(x_1, x_2, x_3)$
and $x_2$ is changed to $y :=(y_1, y_2, y_3)$):
\begin{eqnarray*}
&& |x_1|^2|y_1|^2+|x_2|^2|y_2|^2+|x_3|^2|y_3|^2+2(|x_1|^2|y_2|^2+|x_2|^2|y_3|^2+|x_3|^2|y_1|^2)\\
&& \quad \quad \quad -(x_1\overline{x_2}y_1\overline{y_2}+ \overline{x_1}x_2\overline{y_1}y_2
+x_1\overline{x_3}y_1\overline{y_3}+ \overline{x_1}x_3\overline{y_1}y_3
+x_2\overline{x_3}y_2\overline{y_3}+ \overline{x_2}x_3\overline{y_2}y_3).
\end{eqnarray*}
Since $\cH(x,y) \geq 0$ for all real $x,y$ (see \cite{NieZha16}),
the tensor $\cH$ is $\re$-psd. In fact, it is also $\bC$-psd, because
\[
\baray{rl}
\cH(x,y) =&|x_1|^2|y_1|^2+|x_2|^2|y_2|^2+|x_3|^2|y_3|^2+2(|x_1|^2|y_2|^2+|x_2|^2|y_3|^2+|x_3|^2|y_1|^2) \\
	& -2\big(\mbox{Re}(x_1\overline{x_2}y_1\overline{y_2})+ \mbox{Re}(x_1\overline{x_3}y_1\overline{y_3})+\mbox{Re}(x_2\overline{x_3}y_2\overline{y_3})\big) \\
\ge & |x_1|^2|y_1|^2+|x_2|^2|y_2|^2+|x_3|^2|y_3|^2
       +2(|x_1|^2|y_2|^2+|x_2|^2|y_3|^2+|x_3|^2|y_1|^2)) \\
&  -2(|x_1x_2y_1y_2|+|x_1x_3y_1y_3|+|x_2x_3y_2y_3|) \\
	=& \cH(\hat{x},\hat{y} ) \ge 0,
\earay
\]
where $\hat{x} :=(|x_1|,|x_2|,|x_3|) $ and $\hat{y} :=(|y_1|,|y_2|,|y_3|)$ are real.
\\
(ii) Consider $\cH\in \bC^{[2,2]} $ such that
\[
\cH_{1111}=\cH_{1122}=\cH_{2211}=1, \quad \cH_{1221}=\cH_{2112}=-1
\]
and all other entries are zeros, so
(for cleanness, the variable $x_1$ is changed to
$x := (x_1, x_2)$ and $x_2$ is changed to $y := (y_1, y_2)$):
\[
\cH(x,y)=|x_1|^2|y_1|^2+x_1\overline{x}_2y_1\overline{y}_2+\overline{x}_1x_2\overline{y}_1y_2
-x_1\overline{x}_2\overline{y}_1y_2-\overline{x}_1x_2y_1\overline{y}_2.
\]
When $x,y$ are real,
$
\cH(x,y)=x_1^2y_1^2 \ge 0.
$
This tensor is $\re$-psd but not $\bC$-psd,
because for $x=y=(\sqrt{-1},1)$, $\cH(x,y)=1-1-1-1-1=-3 < 0$.
\end{example}

A $\re$-psd Hermitian tensor is not necessarily $\bC$-psd.
However, for $\re$-Hermitian decomposable tensors,
they are equivalent.

\begin{prop} \label{prop:Rpsd=Cpsd}
For $\cH \in \re_D^{[n_1,\ldots,n_m]}$,
$\cH$ is $\re$-psd if and only if $\cH$ is $\cpx$-psd.
\end{prop}
\begin{proof}
The ``if" direction is obvious.
We prove the ``only if" direction.
For $v^i\in \bC^{n_i} $, write $v^j=x^j+\sqrt{-1}y^j$ with $x^j,y^j\in \re^{n_j}$.
Then, we have
\begin{eqnarray*}
\langle [u^1,\ldots,u^m]_\oh, [v^1,\ldots,v^m]_\oh \rangle
&=& \Pi_{j=1}^m (u^j)^Tv^j \cdot (u^j)^T \overline{v}^j = \Pi_{j=1}^m |(u^j)^T v^j|^2 \\
= \Pi_{j=1}^m (|(u^j)^Tx^j|^2+|(u^j)^Ty^j|^2)
 &=&  \sum_{z^j\in \{x^j,y^j\}}
 \langle [u^1,\ldots, u^m]_\oh, [z^1,\ldots,z^m]_\oh   \rangle.
\end{eqnarray*}
Since $\cH \in \re_D^{[n_1,\ldots,n_m]}$,
it is a sum of real rank-$1$ real Hermitian tensors, so
\[
\left \langle \cH, [v^1,\ldots,v^m]_\oh \right \rangle
= \sum_{z^j\in \{x^j,y^j\}} \left \langle \cH, [z^1,\ldots,z^m]_\oh \right \rangle \ge 0.
\]
If $\cH$ is $\re$-psd, then $\cH$ is also $\cpx$-psd.
\end{proof}

Clearly, $\mathscr{P}_{\F}^{[n_1,\ldots,n_m]}$
is a closed convex cone. As in \cite{BV}, a cone is said to be
{\it solid} if it has nonempty interior; it is said to be
{\it pointed} if it does not contain any line through origin;
a closed convex cone is said to be {\it proper} if
it is both solid and pointed.
The complex cone $\mathscr{P}_{\cpx}^{[n_1,\ldots,n_m]}$
is proper, as mentioned in \cite{Ni19}.
However, the real cone $\mathscr{P}_{\re}^{[n_1,\ldots,n_m]}$ is not proper.
In fact, it is solid but not pointed.
%
%

\begin{proposition} \label{pro:psdHT:proper}
For $m > 1$ and $n_1,\ldots, n_m > 1$,
the cone $\mathscr{P}_{\cpx}^{n_1,\ldots,n_m}$ is proper,
while $\mathscr{P}_\re^{[n_1,\ldots,n_m]}$ is solid but not pointed.
\end{proposition}
\begin{proof}
Let $\mc{I} \in \F^{[n_1,\ldots,n_m]}$ be the identity tensor,
i.e.,
$
\mc{I}(x, \overline{x}) = (x_1^*x_1)\cdots (x_m^*x_m).
$
The conjugate polynomial $\mc{I}(x, \overline{x})$
is positive definite on the spheres $\|x_i \| = 1$.
Thus, for $\eps > 0$ sufficiently small, all Hermitian tensors
$\cH \in \F^{[n_1,\ldots,n_m]}$
with $\| \cH - \mc{I} \| < \eps$ belong to the cone
$\mathscr{P}_\F^{[n_1,\ldots,n_m]}$, for both $\F= \cpx, \re$.
That is, $\mc{I}$ is an interior point,
and hence $\mathscr{P}_\F^{[n_1,\ldots,n_m]}$ is solid.

The complex cone $\mathscr{P}_{\cpx}^{n_1,\ldots,n_m}$ is pointed.
For each $\mH \in \mathscr{P}_{\cpx}^{n_1,\ldots,n_m}
\cap -\mathscr{P}_{\cpx}^{n_1,\ldots,n_m}$,
$\mH(x, \overline{x})$ must be identically zero for all complex $x_i$.
The conjugate polynomial
\[
\mH(x, \overline{x}) = {\sum}_{i_1 \ldots i_m j_1 \ldots j_m}
\mH_{i_1 \ldots i_m j_1 \ldots j_m}
x_{1,i_1} \cdots x_{m,i_m} \overline{x}_{1, j_1} \cdots \overline{x}_{m,j_m}
\]
is identically zero if and only all its coefficients are zero,
i.e., $\mH = 0$. This implies that $\mathscr{P}_{\cpx}^{n_1,\ldots,n_m}$
does not contain any line through origin, i.e., it is pointed.

The real cone $\mathscr{P}_{\re}^{n_1,\ldots,n_m}$ is not pointed.
For $m > 1$ and $n_1,\ldots, n_m > 1$, the set
$\re_D^{[n_1,\ldots,n_m]}$ is a proper subspace of
$\re^{[n_1,\ldots,n_m]}$. Let $C$ be the orthogonal complement of
$\re_D^{[n_1,\ldots,n_m]}$ in $\re^{[n_1,\ldots,n_m]}$.
Then, for all $0 \ne \mc{X} \in C$ and for all $x_j \in \re^{n_j}$,
$
\langle \mc{X}, [x_1, \ldots, x_m]_\oh \rangle = 0
$
because $[x_1, \ldots, x_m]_\oh \in \re_D^{[n_1,\ldots,n_m]}$.
This implies $C \subseteq \mathscr{P}_{\re}^{n_1,\ldots,n_m}$.
So, $\mathscr{P}_{\re}^{n_1,\ldots,n_m}$ contains a line through the origin
and hence it is not pointed.
\end{proof}

\subsection{Hermitian eigenvalues}

For a Hermitian tensor $\cH \in \bC^{[n_1,\ldots,n_m]}$,
consider the sphere optimization problem
\begin{equation}
\label{minH:||x||=1}
\left\{ \baray{rl}
\min  & \cH(x,\overline{x}) := \langle \cH, [x_1, \ldots, x_m]_{\otimes_h} \rangle \\
\mbox{s.t.}  & \|x_i\| = 1, \, x_i \in \cpx^{n_i}, \, i=1, \ldots, m.
\earay \right.
\end{equation}
Since $\cH(x,\overline{x})$ is conjugate quadratic in each $x_k$, we can write it as
\[
\cH(x,\overline{x}) = x_k^* \big( H_k( x_1, \ldots, x_{k-1}, x_{k+1}, \ldots, x_m ) \big) x_k,
\]
where $H_k$ is a Hermitian matrix polynomial.
The $H_k$ is also conjugate quadratic in $x_i$ for all $i \ne k$.
Define the tensor-vector product
\be \label{tv:prod:(k)}
\cH \times_{(k)} \big(x_1, \ldots, x_m \big)
\, := \, H_k( x_1, \ldots, x_{k-1}, x_{k+1}, \ldots, x_m ) x_k.
\ee
The Karush-Kuhn-Tucker (KKT) optimality conditions  for \reff{minH:||x||=1} are
\be  \nn
\cH \times_{(k)} \big(x_1, \ldots, x_m \big)  = \lambda_k x_k,
\, k= 1, \ldots, m.
\ee
where $\lmd_1, \ldots, \lmd_m$ are the Lagrange multipliers.
Since $\cH$ is Hermitian and each $x_k^*x_k=1$,
we must have $\lmd_k \in \re$ and
\[
\lmd_1 = \cdots = \lmd_m = \cH(x,\overline{x}).
\]
They are equal to each other and are all real \cite{Ni19}.

\begin{definition}
\rm  \label{def:HMeig}
(\cite{Ni19})
For $\cH \in \bC^{[n_1,\ldots,n_m]}$,
a scalar $\lmd$ is called a {\em Hermitian eigenvalue} of $\cH$
if there exist complex vectors $u_1, \ldots, u_m$ such that
\be \label{Eq:HermEig}
\left\{\baray{rcl}
\cH \times_{(k)} \big( u_1, \ldots, u_m \big)
&=& \lambda  u_k, \, k = 1,\ldots, m, \\
\|u_1\| =  \cdots = \|u_m\| &=& 1.
\earay \right.
\ee
The $(\lambda; u_1, \ldots, u_m)$
is called a {\it Hermitian eigentuple},
and $u_k$ is called the mode-$k$ {\it Hermitian eigenvector}.
The $(u_1, \ldots, u_m)$
is called a {\it Hermitian eigenvector}.
\end{definition}

For general tensors, similar KKT conditions can be written
and they give unitary eigenvalues \cite{NQB14}.
When $\mH$ is Hermitian, all its Hermitian eigenvalues are real \cite{Ni19}.
The largest (resp., smallest) Hermitian eigenvalue of $\cH$
is the maximum (resp., minimum) value of
$\cH(x,\overline{x})$ over complex spheres $ \| x_i \|=1$.
Therefore, $\mH$ is $\cpx$-psd (resp., $\cpx$-pd)
if and only if all its Hermitian eigenvalues are
nonnegative (resp., greater than zero).
Similarly, for $\cH \in \re^{[n_1,\ldots,n_m]}$,
$\cH$ is $\re$-psd (resp., $\re$-pd)
if and only if all its Hermitian eigenvalues,
which are associated to real eigenvectors, are
nonnegative (resp., greater than zero).

\subsection{Conjugate and Hermitian sum-of-squares}
\label{ssc:CHSOS}

Recall that $\cpx[x]$ denotes the ring of polynomials in
$x = (x_1, \ldots, x_m)$, where each $x_k \in \cpx^{n_k}$,
and $\cpx[x, \overline{x}]$ denotes the ring of conjugate polynomials in $x$.
Psd Hermitian tensors can be detected by SOS decompositions
for conjugate polynomials.

\begin{definition}
\rm \label{df:SOS}
A conjugate polynomial
$f \in \cpx[x, \overline{x}]$ is called a {\it Hermitian sum-of-squares} (HSOS) if
$
f  = |p_1(x)|^2 + \cdots + |p_k(x)|^2
$
for some complex polynomials $p_i \in \cpx[x]$.
It is called a {\it conjugate sum-of-squares} (CSOS) if
$
f =  |q_1(x, \overline{x})|^2 + \cdots + |q_t(x, \overline{x})|^2
$
for some conjugate polynomials $q_i \in \cpx[x, \overline{x}]$.
Denote by $\Sigma [x, \overline{x}]$ (resp., $\Sigma[x]$)
the cone of conjugate (resp., Hermitian) sum-of-squares.
A tensor $\mH \in \bC^{[n_1,\ldots, n_m]}$ is
called HSOS (resp., CSOS) if the corresponding conjugate polynomial
$\cH(x,\overline{x})$ is HSOS (resp., CSOS).
\end{definition}

Clearly, all HSOS and CSOS Hermitian tensors are $\cpx$-psd.
Interestingly, HSOS tensors can be detected by the Hermitian flattening.

\begin{proposition} \label{prop:HSOScond}
For $\mH \in \bC^{[n_1,\ldots, n_m]}$,
$\mH$ is HSOS if and only if $\mathfrak{m}(\mH) \succeq 0$.
\end{proposition}
\begin{proof}
If $\mathfrak{m}(\mH) \succeq 0$, then $\cH$ has the decomposition
\reff{Eq:tensormatrixdecom} with each $\lmd_i >0$. So
$
\cH(x,\overline{x}) = {\sum}_i \lmd_i | \langle \mc{U}_i^*,
x_1 \otimes \cdots \otimes x_m \rangle |^2,
$
and $\cH$ is HSOS. Conversely, if $\cH(x,\overline{x})$ is HSOS, say,
$\cH(x,\overline{x}) =  \sum_{i=1}^k |p_i(x)|^2$,
let $v_i$ be the vectors such that
$
v_i^*(x_1 \boxtimes \cdots \boxtimes  x_m) = p_i(x).
$
Then $\mathfrak{m}(\mH) = v_1 v_1^*+ \cdots + v_k v_k^* \succeq 0$.
\end{proof}

Every HSOS tensor must be CSOS, i.e., $\Sigma[x] \subseteq \Sigma[x,\overline{x}]$.
However, a CSOS tensor is not necessarily HSOS.
The following is such an example.

\begin{example} \label{ex:HSOS}
Let $\cH \in \cpx^{[2,2]}$ be the Hermitian tensor such that
\[
\cH_{1111}=\cH_{2222}=\cH_{1221}=\cH_{2112}=1
\]
and other entries are zeros. Since $\mathfrak{m}(\mH)$
is indefinite, $\cH(x,\overline{x})$ is not HSOS.
However, it is CSOS because
\begin{eqnarray*}
\cH(x,\overline{x})
&=& |(x_1)_1\overline{(x_2)_1}+(x_1)_2\overline{(x_2)_2}|^2.
\end{eqnarray*}
\end{example}

The CSOS Hermitian tensors can be detected by
semidefinite programs~\cite{SDPbk}.
For $\cH \in \bC^{[n_1, \ldots, n_m]}$, if
$
\cH(x,\overline{x}) = |q_1(x, \overline{x})|^2 + \cdots + |q_t(x, \overline{x})|^2
$
for some conjugate polynomials $q_i \in \cpx[x, \overline{x}]$,
then each $q_i$ must have degree $m$ and is linear in $(x_j,\overline{x}_j)$,
for all $j=1,\ldots, m$. Denote the Kronecker product of all vector variables
\be
\mathfrak{b}(x, \overline{x} ) := (x_1,\overline{x}_1)^T \boxtimes \cdots
\boxtimes (x_m,\overline{x}_m)^T.
\ee
For each $q_i$, there exists a coefficient vector $w_i$ such that
$q_i = w_i^* \mathfrak{b}(x, \overline{x} ) $.
The above CSOS decomposition is equivalent to that
\[
\cH(x,\overline{x}) =  \mathfrak{b}(x, \overline{x} )^*
\big( w_1w_1^* + \cdots + w_t w_t^*  \big)
\mathfrak{b}(x, \overline{x} ).
\]

\begin{proposition}
A Hermitian tensor $\cH \in \bC^{[n_1, \ldots, n_m]}$ is CSOS
if and only if there exists a Hermitian matrix $W \succeq 0$ such that
\be \label{H(x)=b(x)*Wb(x)}
\cH(x,\overline{x}) \, = \, \mathfrak{b}(x, \overline{x} )^* \cdot W \cdot
\mathfrak{b}(x, \overline{x} ).
\ee
\end{proposition}
\begin{proof}
If $\cH$ is CSOS, we can just let
$W = w_1w_1^* + \cdots + w_t w_t^*$, for the vectors $w_i$ in the above.
If there exists a psd matrix $W$ satisfying \reff{H(x)=b(x)*Wb(x)},
then there must exist vectors $w_i$ such that
$W = w_1w_1^* + \cdots + w_t w_t^*$,
which implies that $\cH$ is CSOS.
\end{proof}

For a given $\cH$, the set of all psd $W$
satisfying \reff{H(x)=b(x)*Wb(x)} is the intersection of
the psd matrix cone and an affine linear subpsace, i.e.,
the set of all required $W$ is given by linear matrix inequalities.
Therefore, CSOS Hermitian tensors can be detected by
solving semidefinite programs.
We refer to \cite{LasBk15,Lau09,Rez00}
for related work about SOS polynomials.

\subsection{The hierarchy of SOS representations}
\label{ssc:hierSOS}

A Hermitian tensor $\mH$ is $\cpx$-psd if and only if
$\mH(x,\overline{x})$ is nonnegative everywhere. It is well-known that
not every nonnegative polynomial is SOS \cite{Rez00}.
Therefore, not all psd Hermitian tensors are SOS.
However, every nonnegative polynomial is a sum of squares of rational functions.
This motivates us to characterize psd Hermitian tensors
by using products of squares.
For powers $k_1, \ldots, k_m \geq 0$, denote
\be \label{Omega:C:k1tokm}
\Omega_\cpx^{k_1 \ldots k_m} = \Big\{
\mH \in \cpx^{[n_1,\ldots, n_m]}:
|x_1|^{2k_1} \cdots |x_m|^{2k_m} \cdot \cH(x,\overline{x}) \in \Sigma[x]
\Big\} .
\ee
Clearly, if $\mH \in \Omega_\cpx^{k_1 \ldots k_m}$,
then $\mH$ must be $\cpx$-psd.
Each $\Omega_\cpx^{k_1 \ldots k_m}$ is a closed convex cone.
We have the following characterization for $\cpx$-psd tensors.

\begin{theorem}  \label{pd:HSOS:cpx}
If $\mH \in \bC^{[n_1,\ldots, n_m]}$
is $\cpx$-positive definite, then there exist powers $k_1, \ldots, k_m \geq 0$
such that $\cH \in \Omega_{k_1 \ldots k_m}$.
Therefore, we have the containment
\be \label{intPC<=HSOS<=PC}
\mbox{int} \Big( \mathscr{P}_\cpx^{[n_1,\ldots, n_m]}  \Big) \, \subseteq \,
  \bigcup_{  k_1, \ldots, k_m \geq 0 } \Omega_\cpx^{k_1 \ldots k_m}
\, \subseteq \, \mathscr{P}_\cpx^{[n_1,\ldots, n_m]} .
\ee
\end{theorem}
\begin{proof}
When $\mH$ is $\cpx$-positive definite, the real-valued
complex conjugate polynomial $\cH(x,\overline{x})$ is positive on the
complex spheres $x_i^*x_i = 1$.
Consider the ideal $J$ generated by conjugate polynomials
$|x_1|^2-1, \ldots, |x_m|^2 - 1$, in the ring $\cpx[x, \overline{x}]$.
The ideal $J$ is archimedean \cite{Lau09}, since
$m - (|x_1|^2 + \cdots + |x_m|^2) \in J$.
Then, by \cite[Proposition~3.2]{PutSch14},
$\cH(x,\overline{x})$ is HSOS modulo $J$, i.e., there exist complex polynomials
$p_\ell  \in \cpx[x]$ and conjugate polynomials $c_j \in \cpx[x, \overline{x}]$
such that
\[
\cH(x,\overline{x}) = {\sum}_{\ell=1}^N | p_\ell(x) |^2  +
{\sum}_{j=1}^m (| x_j |^2 -1 ) c_j(x, \overline{x}).
\]
Note that $\cH(x,\overline{x})$ is homogeneous quadratic conjugate in each $(x_i, \overline{x}_i)$.
There exist powers $k_1, \ldots, k_m$ such that each $k_i$
is not less than the highest degree of $x_i$ of all polynomials $p_l$.
For each $\ell$, let
\[
q_\ell(x,\overline{x}) := |x_1|^{k_1} \cdots |x_m|^{k_m}
p_\ell( x_1 / |x_1|, \ldots,  x_m / |x_m|).
\]
In the above expression of $\cH(x,\overline{x})$,
we substitute each $x_i$ for $x_i/ | x_i |$, then
\[
\baray{rcl}
P(x, \overline{x}) &:= & |x_1|^{2k_1} \cdots |x_m|^{2k_m} \cH(x_1 / |x_1|, \ldots,  x_m / |x_m|,\overline{x_1} / |x_1|, \ldots,  \overline{x_m} / |x_m|)
\\
 &=& |x_1|^{2k_1-2} \cdots |x_m|^{2k_m-2} \cH(x,\overline{x}) =
{\sum}_{\ell=1}^N |q_\ell(x,\overline{x})|^2 .
\earay
\]
%
%
Since $P(x, \overline{x})$ is also homogeneous in each $(x_i, \overline{x}_i)$ with degree $2k_i$,
each $q_l(x,\overline{x})$ must have the same degree $k_i$ in $(x_i, \overline{x}_i)$.
Write each $q_\ell$ in the form as
\[
q_\ell(x,\overline{x}) = \sum_{0 \le s_i \le k_i }
|x_1|^{s_1}\cdots |x_m|^{s_m} \cdot g_l^{s_1,\ldots,s_m}(x)
\]
where $g_l^{s_1,\ldots,s_m}(x)$ is a complex polynomial.
The degree of $x_i$ in each term of $g_l^{s_1,\ldots,s_m}(x)$
must be $k_i-s_i$, so
\begin{eqnarray*}
& &|q_\ell(x,\overline{x})|^2 =
\left|
\sum_{ 0\le s_i \le k_i }
|x_1 |^{s_1} \cdots  | x_m |^{s_m} g_l^{s_1,\ldots,s_m}(x)
\right|^2\\
	&=& \sum_{0\le t_i \le k_i}\sum_{0\le s_i \le k_i}
 | x_1 |^{s_1}\cdots  | x_m |^{s_m} g_l^{s_1,\ldots,s_m}(x)  |x_1 |^{t_1}\cdots | x_m |^{t_m} \overline{g_l^{t_1,\ldots,t_m}(x)} \\
	&=& \sum_{0\le t_i \le k_i}\sum_{0\le s_i \le k_i} |x_1|^{s_1+t_1}\cdots |x_m|^{s_m+t_m} g_l^{s_1,\ldots,s_m}(x) \overline{g_l^{t_1,\ldots,t_m}(x)}.
\end{eqnarray*}
The degrees of $x_i$ and $\overline{x}_i$ of
$g_l^{s_1,\ldots,s_m}(x)\overline{g_l^{t_1,\ldots,t_m}(x)}$
are $k_i-s_i,k_i-t_i$ respectively.
The degrees of $x_i$ and $\overline{x}_i$ in $P(x,\overline{x})$ must match,
so $g_l^{s_1,\ldots,s_m}(x)\overline{g_l^{t_1,\ldots,t_m}(x)}$
has the same degree of $x_i$ and $\overline{x}_i$ if and only $s_i=t_i$. Thus
the terms like $|x_1|^{s_1+t_1}\cdots |x_m|^{s_m+t_m} g_l^{s_1,\ldots,s_m}(x) \overline{g_l^{t_1,\ldots,t_m}(x)}$ with
$(s_1,\ldots,s_m)\neq (t_1,\ldots,t_m) $
will be canceled after the expansion. Therefore
\begin{eqnarray*}
P(x, \overline{x} ) &=& \sum_{\ell=1}^N \sum_{0\le t_i \le k_i}\sum_{0\le s_i \le k_i} |x_1|^{s_1+t_1}\cdots |x_m|^{s_m+t_m}
g_l^{s_1,\ldots,s_m}(x) \overline{g_l^{t_1,\ldots,t_m}(x)} \\
&=& \sum_{\ell=1}^N \sum_{0\le s_i \le k_i} |x_1|^{2s_1}\cdots |x_m|^{2s_m} |g_l^{s_1,\ldots,s_m}(x)|^2 ,
\end{eqnarray*}
which means that $P(x, \overline{x} )$ is HSOS.
\end{proof}

For $\re$-psd Hermitian tensors, we have a similar conclusion.
For each $\mH \in \re^{[n_1,\ldots, n_m]}$, $\mH(x,x) = \mH(x,\overline{x})$
for real $x$. So, $\mH$ is $\re$-psd
if and only if $\mH(x,x)\geq 0$ for all real $x$.
Note that $\mH(x,x)$ is a real multi-quadratic homogeneous polynomial.
The classical Positivstellensatz \cite{LasBk15,Lau09,Put93,Rez00}
for real positive polynomials can be used to characterize $\re$-psd tensors.
For powers $k_i \geq 0$, we can similarly define the cone
(denote by $\Sigma[x]_\re$ the cone of real SOS polynomials
in $\re[x]$, i.e., $\Sigma[x]_\re$ is the cone generated by
squares $q^2$, for $q \in \re[x]$)
\be \label{Omega:re:k1tokm}
\Omega_\re^{k_1 \ldots k_m} = \Big\{
\mH \in \re^{[n_1,\ldots, n_m]}:
(x_1^Tx_1)^{k_1} \cdots (x_m^Tx_m)^{k_m} \cdot \cH(x,x) \in \Sigma[x]_\re
\Big\} .
\ee
Clearly, if $\mH \in \Omega_\re^{k_1 \ldots k_m}$,
then $\mH$ must be $\re$-psd.
Each $\Omega_\re^{k_1 \ldots k_m}$ is a closed convex cone.
We have the following characterization for $\re$-psd tensors.

\begin{theorem} \label{pd:HSOS:real}
If $\mH \in \re^{[n_1,\ldots, n_m]}$
is $\re$-positive definite, then there exist powers $k_1, \ldots, k_m \geq 0$
such that $\cH \in \Omega_\re^{k_1 \ldots k_m}$.
Therefore, we have
\be
\mbox{int} \Big( \mathscr{P}_\re^{[n_1,\ldots, n_m]}  \Big) \, \subseteq \,
  \bigcup_{  k_1, \ldots, k_m \geq 0 } \Omega_\re^{k_1 \ldots k_m}
\, \subseteq \, \mathscr{P}_\re^{[n_1,\ldots, n_m]} .
\ee
\end{theorem}

The proof for Theorem~\ref{pd:HSOS:real} is the same as the one for
Theorem~\ref{pd:HSOS:cpx}. In fact, the proof is easier
because it deals with real polynomials instead of conjugate polynomials.
Each product $(x_1^Tx_1)^{k_1} \cdots (x_m^Tx_m)^{k_m} \cdot \cH(x,x)$
is a real polynomial in $x$. The conclusion can be implied
by classical results about real positive polynomials over compact semialgebraic sets
\cite{LasBk15,Lau09,Put93,Rez00}.
For cleanness of the paper, we omit the proof.

\section{Separable Hermitian Tensors}
\label{sc:sepent}

A basic topic in quantum physics is tensor entanglement.
It requires to decide whether or not a given Hermitian tensor
can be written as a sum of rank-$1$ Hermitian tensors with positive coefficients.
This leads to the concept of {\it separable} tensors.

\begin{definition} \rm
\cite{Ni19} A Hermitian tensor $\cH \in \bC^{[n_1, \ldots, n_m]}$
is called {\it separable} if
\be \label{Eq:separabletensor}
\cH = [u_1^1, \ldots, u_1^m]_{\otimes h} + \cdots +
 [u_r^1, \ldots, u_r^m]_{\otimes h}
\ee
for some vectors $u_i^j\in \mathbb{C}^{n_j}$. When it exists,
\reff{Eq:separabletensor} is called a
{\it positive $\cpx$-Hermitian decomposition}
and $\cH$ is called {\it $\cpx$-separable}.
Moreover, if each $u_i^j$ in \reff{Eq:separabletensor}
is real, then $\cH$ is called {\it $\re$-separable}
and \reff{Eq:separabletensor} is called a
{\it positive $\re$-Hermitian decomposition}.
\end{definition}

Let $\F = \cpx$ or $\re$. The set of $\F$-separable tensors in
$\F^{[n_1, \ldots, n_m]}$ is denoted as $\mathscr{S}_{\F}^{[n_1,\ldots,n_m]}$.
%
%
The decomposition~\reff{Eq:separabletensor}
is equivalent to that
\[
\cH(x,\overline{x}) = {\sum}_{i=1}^r | (u_i^1)^*x_1|^2 \cdots | (u_i^m)^*x_m|^2 .
\]
All $\F$-separable tensors must be HSOS.
To be $\re$-separable, a tensor must be $\re$-Hermitian decomposable.
The following is the relationship between
$\cpx$-separability and $\re$-separability.

\begin{lemma} \label{lemma: R-decom sep}
For $\mH \in \re_D^{[n_1, \ldots, n_m]}$,
$\mH$ is $\re$-separable if and only if it is $\bC$-separable.
\end{lemma}
\begin{proof}
The ``only if" direction is obvious.
We prove the ``if" direction.
Assume $\cH$ is $\bC$-separable, then \reff{Eq:separabletensor}
holds for some complex vectors $u_i^j$. Let
$s^j_i :=\mbox{Re}(u_i^j)$ and $t^j_i :=\mbox{Im}(u_i^j)$.
For all {\em real} vector variables $x_i \in \re^{n_i}$,
the inner product $\langle [u_i^1,\ldots,u_i^m]_\oh, [x_1,\ldots,x_m]_\oh  \rangle
=\prod_{j=1}^m |(u_i^j)^* x^j|^2 $,
which can be expanded as
\[
\prod_{j=1}^m \Big(|(s_i^j)^Tx_j|^2+|(t_i^j)^Tx_j|^2 \Big)
=\sum_{z_i^j\in \{s_i^j,t_i^j\}}
 \langle [z_i^1,\ldots,z_i^m]_\oh, [x_1,\ldots,x_m]_\oh  \rangle .
\]
The equation \reff{Eq:separabletensor} implies that, for all real vectors $x_i$,
\[
\langle \cH, [x_1,\ldots,x_m]_\oh \rangle =
\sum_{i=1}^r\sum_{z_i^j\in \{s_i^j,t_i^j\}}
\langle [z_i^1,\ldots,z_i^m]_\oh, [x_1,\ldots,x_m]_\oh \rangle.
\]
Since $\cH$ is $\re$-separable, by Lemma~\ref{rH:poly=decomp},
$
\cH = \sum_{i=1}^r\sum_{z_i^j\in \{s_i^j,t_i^j\}} [z_i^1,\ldots,z_i^m]_\oh.
$
Hence, $\cH$ is also $\re$-separable.
\end{proof}

\subsection{The dual relationship}

The complex separable tensor cone $\mathscr{S}_{\cpx}^{[n_1,\ldots,n_m]}$
is dual to $\mathscr{P}_{\cpx}^{[n_1,\ldots,n_m]}$,
as noted in~\cite{Ni19}. The duality also holds for the real case.
Let $\F = \cpx$ or $\re$. By the definition (see \cite{BrksCOT}),
the dual cone of $\mathscr{S}_{\F}^{[n_1,\ldots,n_m]}$ is the set
\[
\Big( \mathscr{S}_{\F}^{[n_1,\ldots,n_m]} \Big)^\star :=
\Big\{ X \in \F^{[n_1,\ldots,n_m]}:
\langle X, Y \rangle \geq 0 \,\forall \, Y \in
\mathscr{S}_{\F}^{[n_1,\ldots,n_m]} \Big\}.
\]
Recall that a closed convex cone is proper if it is solid
(has nonempty interior) and pointed (does not contain any line through the origin).
The complex cone $\mathscr{S}_{\cpx}^{[n_1,\ldots,n_m]}$
is proper \cite{Ni19}, but
$\mathscr{S}_{\re}^{[n_1,\ldots,n_m]}$ is not.
%
%

\begin{theorem} \label{thm:sep:dual}
For $\F = \re,\cpx$,  the cone $\mathscr{S}_{\F}^{[n_1,\ldots,n_m]}$
is dual to $\mathscr{P}_{\F}^{[n_1,\ldots,n_m]}$, i.e.,
\be \label{psd:dual:sep}
\Big( \mathscr{S}_{\F}^{[n_1,\ldots,n_m]} \Big)^\star =
\mathscr{P}_{\F}^{[n_1,\ldots,n_m]}, \quad
\Big( \mathscr{P}_{\F}^{[n_1,\ldots,n_m]} \Big)^\star =
\mathscr{S}_{\F}^{[n_1,\ldots,n_m]}.
\ee
Moreover, the complex cone $\mathscr{S}_{\cpx}^{[n_1,\ldots,n_m]}$ is proper,
while the real one $\mathscr{S}_{\re}^{[n_1,\ldots,n_m]}$ is not proper.
In fact, $\mathscr{S}_{\re}^{[n_1,\ldots,n_m]}$ is pointed but not solid.
\end{theorem}
\begin{proof}
Observe that $\mathscr{S}_{\F}^{[n_1,\ldots,n_m]}$
equals the conic hull of the compact set
\be \label{set:U}
U := \Big( [u_1, \ldots, u_m]_{\otimes h}:
u_i \in \F^{n_i}, \| u_i \| = 1  \Big),
\ee
so it is a closed convex cone \cite{BrksCOT}.
A tensor $X \in {\F}^{[n_1,\ldots,n_m]}$ belongs to the dual cone of
$\mathscr{S}_{\F}^{[n_1,\ldots,n_m]}$
if and only if
$
\langle X, [u_1, \ldots, u_m]_{\otimes h} \rangle \geq 0
$
for all $u_i \in \F^{n_i}$,
which is equivalent to that $X$ is $\F$-psd.
Therefore, the dual cone of $\mathscr{S}_{\F}^{[n_1,\ldots,n_m]}$
is $\mathscr{P}_{\F}^{[n_1,\ldots,n_m]}$.
Since $\mathscr{S}_{\F}^{[n_1,\ldots,n_m]}$
and $\mathscr{P}_{\F}^{[n_1,\ldots,n_m]}$ are both closed convex cones,
the dual cone of $\mathscr{P}_{\F}^{[n_1,\ldots,n_m]}$
is also equal to $\mathscr{S}_{\F}^{[n_1,\ldots,n_m]}$,
by the bi-duality theorem \cite{BrksCOT}.
Hence, the dual relationship \reff{psd:dual:sep} holds.
By Proposition~\ref{pro:psdHT:proper}, the cone
$\mathscr{P}_{\cpx}^{[n_1,\ldots,n_m]}$ is proper,
while $\mathscr{P}_{\re}^{[n_1,\ldots,n_m]}$ is solid but not pointed.
By the duality, $\mathscr{S}_{\cpx}^{[n_1,\ldots,n_m]}$ is also proper,
while $\mathscr{S}_{\re}^{[n_1,\ldots,n_m]}$
is pointed but not solid \cite{BTN01}.
\end{proof}

Theorem~\ref{thm:sep:dual} tells that a Hermitian tensor is $\F$-separable
if and only if it belongs to the dual cone of
$\mathscr{P}_{\F}^{[n_1,\ldots,n_m]}$. 
Therefore, for $\A \in \F^{[n_1,\ldots,n_m]}$,
if there exists $\B \in \F^{[n_1,\ldots,n_m]}$ such that
$\B(x,\overline{x})  \in \Sigma[x, \overline{x}]$ and $\langle \A, \B \rangle < 0$,
then $\A$ is not $\F$-separable. For instance,
consider the Hankel tensor $\mA \in \bC^{[2,2]}$ such that
$\mA_{ijkl} =  i+j+k+l$
for all $i,j,k,l$. Let $\mB$ be the Hermitian tensor such that
\[
  \langle \B, [x_1,x_2]_{\otimes h} \rangle  
   = |x_{11}x_{21}-\frac{5}{6}x_{11}x_{22} |^2  .
\]
Since $\B(x) \in \Sigma[x]$ and
$
\langle \A, \B \rangle = -\frac{1}{6}<0,
$
$\A$ is not $\F$-separable for $\F= \cpx,\re$.

\subsection{Reformulations for separability}

An important computational task is to determine whether or not
a Hermitian tensor is separable.
If it is, we need a positive Hermitian decomposition.
This is an interesting future work.

Let $\F = \cpx,\re$. In the proof of Theorem~\ref{thm:sep:dual},
we have seen that the $\F$-separable Hermitian tensor cone
$\mathscr{S}_{\F}^{[n_1,\ldots,n_m]}$ equals the conic hull
of the compact set $U$, that is, ($\mbox{cone}$ denotes the conic hull)
\be
\mathscr{S}_{\F}^{[n_1,\ldots,n_m]} \, = \,
\mbox{cone}\Big( [u_1, \ldots, u_r]_{\otimes h}:
u_i \in \F^{n_i}, \| u_i \| = 1 \Big) .
\ee
Equivalently, we have $\mA \in \mathscr{S}_{\F}^{[n_1,\ldots,n_m]}$
if and only if there exist positive scalars $\lmd_i >0$
and unit length vectors $u_i^j \in \F^{n_j}$ such that
\be \label{meas:mA=lmd_i:otimes:uij}
\mA = {\sum}_{i=1}^r \lmd_i [u_i^1, \ldots, u_i^m]_{\otimes h}.
\ee
If we let
$\mu := {\sum}_{i=1}^r \lmd_i \dt_{(u_i^1, \ldots, u_i^m)}$
be the weighted sum of Dirac measures,
then \reff{meas:mA=lmd_i:otimes:uij} is equivalent to
\be \label{integral:mA=[x]dmu}
\mA = \int  [x_1, \ldots, x_m]_{\otimes h} \mathtt{d} \mu.
\ee
The support $\supp{\mu}$ of the measure $\mu$ is contained in the multi-sphere
\[
\mathbb{S}_\F^{n_1,\ldots,n_m}  \, := \,
\{(x_1, \ldots, x_m) \in \F^{n_1} \times \cdots \F^{n_m}:
\| x_1 \| = \cdots \| x_m \| = 1 \}.
\]
Interestingly, if there is a Borel measure $\mu$ supported in
$\mathbb{S}_\F^{n_1,\ldots,n_m}$, then there must exist
$\lmd_i > 0$ and unit length vectors
$u_i^j$ satisfying \reff{meas:mA=lmd_i:otimes:uij}.
This can be implied by the proof of Theorem~5.9 of \cite{Lau09}.
Therefore, we have the following theorem.

\begin{theorem} \label{thm:TMP:sepa}
For $\F = \cpx$ or $\re$, a tensor
$\mA \in \F^{[n_1,\ldots,n_m]}$ is $\F$-separable if and only if
there exists a Borel measure $\mu$ such that \reff{integral:mA=[x]dmu}
holds and $\supp{\mu} \subseteq \mathbb{S}_\F^{n_1,\ldots,n_m}$.
\end{theorem}

The task of checking existence of $\mu$ in
Theorem~\ref{thm:TMP:sepa} is a truncated moment problem.
We refer to \cite{LasBk15,Lau09,NieZha16,linmomopt,ATKMP}
for related work. Interestingly, separable Hermitian tensors
can also be characterized by the Hermitian flattening map $\mathfrak{m}$.
As in \reff{H=qq*:rk1}, the decomposition~\reff{meas:mA=lmd_i:otimes:uij}
is equivalent to that
\be \label{krondc:m(A)}
\mathfrak{m}(\mA) \,= \, {\sum}_{i=1}^r \lmd_i
\big(u_i^1 (u_i^1)^*\big) \boxtimes \cdots \boxtimes \big(u_i^m (u_i^{m})^*\big).
\ee
The Theorem~\ref{thm:TMP:sepa} immediately implies the following.

\begin{theorem} \label{thm:HerFl:sepa}
For $\F = \cpx$ or $\re$,
a tensor $\mA \in \F^{[n_1,\ldots,n_m]}$
is $\F$-separable if and only if
there exist Hermitian psd matrices $0 \preceq B_{ij} \in \F^{n_j \times n_j}$,
for $i=1,\ldots, s$ and $j=1,\ldots, m$, such that
\be  \label{sepa:A=sum:Bij:ot}
\mathfrak{m}(\mA) \, = \, {\sum}_{i=1}^s
B_{i1} \boxtimes \cdots \boxtimes B_{im}.
\ee
\end{theorem}

The smallest integer $s$ in \reff{sepa:A=sum:Bij:ot}
is called the {\it $\F$-psd rank} for the tensor $\mA$.
How to determine $\F$-psd ranks is mostly an open question.

\begin{example}
Consider the tensor $\mA \in \bC^{[2,2]}$ with the Hermitian flattening
\[
\mathfrak{m}(\mA) =
\left(\baray{rrrr}
  5   &  -4   &   1   &  -5  \\
 -4   &  21   &  -5   &   7  \\
  1   &  -5   &   3   &  -3  \\
 -5   &   7   &  -3   &  13  \\
\earay \right).
\]
It is $\re$-separable, because
\[
\mathfrak{m}(\mA) =
\bpm 2 & -1 \\ -1  & 1 \epm \boxtimes  \bpm 1 & 1 \\ 1 & 3 \epm +
\bpm 3 & 2 \\  2 & 2 \epm \boxtimes \bpm 1 & -2 \\ -2 & 5 \epm.
\]
The $\re$-psd rank is $2$, since $\mA$ does not have a decomposition
like \reff{sepa:A=sum:Bij:ot} for $s=1$. 
\end{example}

\section{Conclusions and future work}
\label{sc:con}

This paper studies Hermitian tensors, Hermitian decompositions, and related topics.
Every complex Hermitian tensor is a sum of complex Hermitian rank-$1$ tensors.
However, this is not true for the real case.
A real Hermitian tensor is not a sum of real rank-$1$ Hermitian tensors,
unless it belongs to a proper subspace.
We study basic properties about Hermitian decompositions and Hermitian ranks.
For canonical basis tensors, we have determined their Hermitian ranks
as well as the rank decompositions.
For real Hermitian tensors, we give a full characterization
for them to have Hermitian decompositions over the real field.
In addition to classical flattening,
there are two special types of matrix flattening for Hermitian tensors:
the Hermitian flattening and Kronecker flattening.
They may give different lower bounds for Hermitian ranks.
We give SOS characterizations for
psd Hermitian tensors. Separable Hermitian tensors
can be formulated as truncated moment problems over multi-spheres.
The cones of psd and separable Hermitian tensors
are dual to each other.

A basic question is to determine
Hermitian ranks, as well as the rank decompositions.
For general Hermitian tensors,
we do not know how to do that. This is an important future work.
We also have the notions of typical and generic Hermitian ranks.
To the best of the authors' knowledge,
the following question is mostly open.

\begin{problem}
For $m>1$ and $n_1, \ldots, n_m > 1$,
what is the generic Hermitian rank of $\bC^{[n_1,\ldots,n_m]}$?
Does $\bC^{[n_1,\ldots,n_m]}$ have a unique typical Hermitian rank?
If not, what is the range of typical Hermitian ranks?
For what cases of $m$ and $n_1, \ldots, n_m$,
does the expected Hermitian rank of $\bC^{[n_1,\ldots,n_m]}$
equal the generic Hermitian rank?
\end{problem}

Real Hermitian tensors are of strong interests in applications.
They may not have real Hermitian decompositions,
unless they lie in the subspace $\re_D^{[n_1,\ldots,n_m]}$.
It is expected that there is $\mH \in \re_D^{[n_1,\ldots,n_m]}$
such that $\hrank_{\re} (\mH) > \hrank (\mH)$.
However, such an explicit $\mH$ is not known to the authors.
So we pose the following question.

\begin{problem}  \label{prob:reD}
For what $\mH \in \re_D^{[n_1,\ldots,n_m]}$
does $\hrank_{\re} (\mH) > \hrank (\mH)$?
Does $\re_D^{[n_1,\ldots,n_m]}$ have an open subset
$T$ such that $\hrank_{\re} (\mH) = \hrank (\mH)$
for all $\mH \in T$?
\end{problem}

We remark that the answer to the second part of Problem~\ref{prob:reD}
is affirmative for the case $m=2$ and $n_1=n_2=2$.
Consider the identity tensor $\mathcal{I} \in \re_D^{[2,2]} $.
Its Hermitian flattening matrix is $I_4$.
Let $T = \{\cH\in \re_D^{[2,2]}: \|\cH-\mathcal{I}\|< 1 \} $, an open subset.
For all $\cH\in T$,
\[
\mathfrak{m} (\cH) = \begin{pmatrix}
		A & C \\
		C & B
 \end{pmatrix},
A \succ 0,
\]
since $\|A-I_2\|_2 \le \|\cH-\mathcal{I} \| < 1 $.
It holds that $\hrank_\re(\cH)=\hrank(\cH) $
for all $\cH \in T$, by Theorem~\ref{thm:RD22:rank}.
We are not sure if the same result holds
for general cases of $n_1, \ldots, n_m$.

For a separable Hermitian tensor $\mA \in \mathscr{S}_{\F}^{[n_1,\ldots,n_m]}$,
its $\F$-psd rank is the smallest integer $s$ in \reff{sepa:A=sum:Bij:ot}.
An important future work is to determine
$\F$-psd ranks for separable Hermitian tensors.

\begin{problem}
For $\mA \in \mathscr{S}_{\F}^{[n_1,\ldots,n_m]}$,
how do we determine its $\F$-psd rank?
\end{problem}

\bigskip
\noindent
{\bf Acknowledgement}
The authors are partially supported by the NSF grant DMS-1619973.
They would like to thank Lek-Heng Lim, Guyan Ni and the anonymous referees
for fruitful suggestions on improving the paper.








\begin{thebibliography}{99}


\bibitem{ardila2018measuring}
{\sc L.~Ardila, M.~Heyl and A.~ Eckardt},
{\em Measuring the single-particle density matrix for fermions and
hard-core bosons in an optical lattice},
Physical review letters, 121 (2018), no.~26, pp.~260--401.


\bibitem{BTN01}
{\sc A. Ben-Tal and A. Nemirovski},
{\em Lectures on Modern Convex Optimization: Analysis, Algorithms, and Engineering Applications},
MPS-SIAM Series on Optimization, SIAM, Philadelphia, 2001.


\bibitem{BBCM13}
{\sc A.~Bernardi, J.~Brachat, P.~Comon, and B.~Mourrain},
{\em General tensor decomposition, moment matrices and applications},
{J. Symbolic Comput.}, 52 (2013), pp.~51--71.


\bibitem{BerBlekOtt18}
{\sc A.~Bernardi, G.~Blekherman, and G.~Ottaviani},
{\em On real typical ranks},
 Boll. Unione Mat. Ital., 11 (2018), no.~3, pp.~293--307.





\bibitem{BrksCOT}
{\sc D.~P.~Bertsekas},
{\em Convex Optimization Theory},
Athena Scientific, 2009.



\bibitem{BPT13}
{\sc G.~Blekherman, P.~Parrilo and R.~Thomas (eds.)},
{\em Semidefinite optimization and convex algebraic geometry},
MOS-SIAM series on Optimization, SIAM, Philadelphia, PA, 2013.


\bibitem{BlekTeit15}
{\sc G.~Blekherman and Z.~Teitler},
{\em On maximum, typical and generic ranks},
Math. Ann.,
362 (2015), no.~3--4, pp.~1021--1031.


\bibitem{Blek15}
{\sc G.~Blekherman},
{\em Typical real ranks of binary forms},
 Found. Comput. Math.,
15 (2015), no.~3, pp.~793--798.

\bibitem{blum2012density}
{\sc K.~Blum},
{\em Density matrix theory and applications},
Springer Science \& Business Media, 2012.

%
%

\bibitem{BV}
{\sc S.~Boyd and L.~Vandenberghe},
{\em Convex Optimization},
Cambridge University Press,  2004.



\bibitem{BCMT10}
{\sc J.~Brachat, P.~Comon, B.~Mourrain, and E.~Tsigaridas},
{\em Symmetric tensor decomposition},
 Linear Algebra Appl., 433 (2010), no.~11--12, pp.~1851--1872.


\bibitem{BreVan18}
{\sc P.~Breiding and N.~Vannieuwenhoven},
{\em A Riemannian trust region method for the canonical tensor rank approximation problem},
 SIAM J. Optim., 28 (2018), no.~3, pp.~2435--2465.


%
%

%
%


%
%
%
%


\bibitem{calderaro2018direct}
{\sc L.~Calderaro, G.~Foletto, D.~ Dequal, P.~Villoresi and G.~Vallone},
{\em Direct reconstruction of the quantum density matrix by strong measurements},
Physical review letters, 121 (2018), no.~23, pp.~230--501.



\bibitem{ChOtVan17}
{\sc L.~Chiantini, G.~Ottaviani, and N.~Vannieuwenhoven},
{\em Effective criteria for specific identifiability of tensors and forms},
SIAM J. Matrix Anal. Appl.,
38 (2017), pp.~656--681.


\bibitem{Comon2008}
{\sc P.~Comon, G.~Golub, L.-H.~Lim, and B.~Mourrain},
{\em Symmetric tensors and symmetric tensor rank},
SIAM J. Matrix Anal. Appl., 30 (2008), no.~3, pp.~1254--1279.


\bibitem{CLQY18}
{\sc P.~Comon, L.-H.~Lim, Y.~Qi and K.~Ye},
{\em Topology of tensor ranks},
Advances in Mathematics, vol.~367, pp. 107-128, 2020.

%
%




\bibitem{DLMO07}
{\sc G.~Dahl, J.~M.~Leinaas, J.~Myrheim, and E.~Ovrum},
{\em A tensor product matrix approximation problem in quantum physics},
Linear Algebra and its Applications,
420 (2007),  pp. 711--725.





%
%

\bibitem{dLMV04}
{\sc L.~De~Lathauwer, B.~De~Moor, and J.~Vandewalle},
{\em Computation of the canonical decomposition by means
of a simultaneous generalized Schur decomposition},
 SIAM J. Matrix Anal. Appl., 26 (2004), no.~2, pp.~295--327.


\bibitem{dLa06}
{\sc L.~De~Lathauwer},
{\em A link between the canonical decomposition in
multilinear algebra and simultaneous matrix diagonalization},
SIAM J. Matrix Anal. Appl., 28 (2006), no.~3, pp.~642--666.



\bibitem{DeSLim08}
{\sc V. De~Silva and L.-H.~Lim},
{\em Tensor rank and the ill-posedness of the best low-rank approximation problem},
SIAM. J. Matrix Anal. Appl., 30 (2008), no.~3, pp.~1084--1127.



\bibitem{DFLW17}
{\sc H.~Derksen, S.~Friedland, L.-H.~Lim and L.~Wang},
{\em Theoretical and computational aspects of entanglement},
\url{arXiv:1705.07160}, preprint, 2017.

%
%

\bibitem{domanov15}
{\sc I.~Domanov, and L.~De~Lathauwer},
{\em Generic uniqueness conditions for the canonical polyadic decomposition and INDSCAL},
SIAM J. Matrix Anal. Appl., 36 (2015), no.~4, pp.~1567--1589.

\bibitem{FuJiangLi18}
{\sc T.~Fu, B.~Jiang and Z.~Li},
{\em On decompositions and approximations of conjugate partial-symmetric complex tensors},
\url{arXiv:1802.09013}, preprint, 2018.


\bibitem{GalMel}
{\sc F.~Galuppi and M.~Mella},
{\em Identifiability of homogeneous polynomials and Cremona Transformations}, J. Reine Angew. Math.,
757 (2019), pp.~279--308.


%
%

\bibitem{JiangLiZh16}
{\sc B.~Jiang, Z.~Li, and S.~Zhang},
{\em Characterizing real-valued multivariate complex polynomials
and their symmetric tensor representations},
SIAM J. Matrix Anal. Appl., 37 (2016), no.~1, pp.~ 381--408.


\bibitem{KolBad09}
{\sc T.~Kolda and B.~Bader},
{\em Tensor decompositions and applications},
SIAM Rev., 51 (2009), no.~3, pp.~455--500.




\bibitem{Kru77}
{\sc J.~Kruskal},
{\em Three-way arrays: rank and uniqueness of trilinear decompositions, with application to arithmetic complexity and statistics},
Lin. Alg. Appl., 18 (1977), no.~2, pp.~95--138.



%
%

\bibitem{Land12}
{\sc J.~Landsberg},
{\em Tensors: Geometry and Applications},
Grad. Stud. Math., Providence, 2012.




\bibitem{LasBk15}
{\sc J.B.~Lasserre},
{\em Introduction to Polynomial And Semi-Algebraic Optimization},
Cambridge University Press, Cambridge, 2015.



%
%



\bibitem{Lau09}
{\sc M.~Laurent},
{\em Sums of squares, moment matrices and optimization over polynomials},
Emerging Applications of Algebraic Geometry of IMA Volumes in
Mathematics and its Applications,
149 (2009), pp. 157--270.


\bibitem{LNSU18}
{\sc Z.~Li, Y.~Nakatsukasa, T.~Soma and A.~Uschmajew},
{\em On orthogonal tensors and best rank-one approximation ratio},
SIAM J. Matrix Anal. Appl.,
39 (2018), no.~1, pp.~400--425.


\bibitem{Lim13}
{\sc L.-H.~Lim},
{\em Tensors and hypermatrices, in: L. Hogben (Ed.)},
Handbook of linear algebra, 2nd Ed.,
CRC Press, Boca Raton, 2013.



\bibitem{NQB14}
{\sc G.~Ni, L.~Qi, and M.~Bai},
{\em Geometric measure of entanglement and U-eigenvalues of tensors},
SIAM J. Matrix Anal. Appl., 35 (2014), no.~1, pp.~73--87.



\bibitem{Ni19}
{\sc G.~Ni},
{\em Hermitian tensor and quantum mixed state},
\url{arXiv:1902.02640 [quant-ph]},
preprint, 2019.


%
%



\bibitem{NieZha16}
{\sc J.~Nie and X.~Zhang},
{\em Positive maps and separable matrices},
SIAM J. Optim., 26 (2016), no.~2, pp.~1236--1256.


\bibitem{Nie-GP}
{\sc J.~Nie},
{\em Generating polynomials and symmetric tensor decompositions},
Found. Comput. Math., 17 (2017), no.~2, pp.~423--465.



\bibitem{linmomopt}
{\sc J.~Nie},
{\em Linear optimization with cones of moments and nonnegative polynomials}, Math. Program., 153 (2015), pp.~247--274.


\bibitem{ATKMP}
{\sc J.~Nie},
{\em The $\A$-truncated $K$-moment problem},
Found. Comput. Math., 14 (2014), no.~6, pp.~1243--1276.


\bibitem{NieYe19}
{\sc J.~Nie and K.~Ye},
{\em Hankel tensor decompositions and ranks},
SIAM J. Matrix Anal. Appl.,
40 (2019), no.~2, pp.~486--516.


\bibitem{OedOtt13}
{\sc L.~Oeding and G.~Ottaviani},
{\em Eigenvectors of tensors and algorithms for waring decomposition},
J. Symbolic Comput., 54 (2013), pp.~9--35.


\bibitem{Put93}
{\sc M.~Putinar},
{\em Positive polynomials on compact semi-algebraic sets},
Ind. Univ. Math. J., 42 (1993), pp.~203--206.



\bibitem{PutSch14}
{\sc M.~Putinar and C.~Scheiderer},
{\em Quillen property of real algebraic varieties},
 M\"{u}nster J. Math. 7 (2014), pp.~671-696.



\bibitem{QiLuo17}
{\sc L.~Qi and Z.~Luo},
{\em Tensor analysis: Spectral theory and special tensors},
SIAM, Philadelphia, 2017.


\bibitem{QiZhangNi18}
{\sc L.~Qi, G.~Zhang, and G.~Ni},
{\em How entangled can a multi-party system possibly be?},
Physics Letters A,
382 (2018), no.~22, pp.~1465-1471.


\bibitem{Rez00}
{\sc B.~Reznick},
{\em Some concrete aspects of Hilbert's $17^{th}$ problem},
Contemp. Math., 253 (2000), pp.~251--272.

%
%

\bibitem{SB00}
{\sc N.~Sidiropoulos and R.~Bro},
{\em On the uniqueness of multilinear decomposition of $N$-way arrays},
J. Chemometrics, 14 (2000), no,~3, pp.~229--239.


\bibitem{SvBdL13}
{\sc L.~Sorber, M.~Van~Barel, and L.~De~Lathauwer},
{\em Optimization-based algorithms for tensor decompositions:
canonical polyadic decomposition, decomposition in
rank-$(L_r,L_r,1)$ terms and a new generalization},
SIAM J. Optim., 23 (2013), no.~2, pp.~695--720.

%
%

%
%
%

\bibitem{tensorlab}
{\sc N. Vervliet, O. Debals, L. Sorber, M. Van Barel, and L.~De~Lathauwer},
{\em Tensorlab 3.0}, March 2016, \url{http://www.tensorlab.net}.




\bibitem{SDPbk}
{\sc H. Wolkowicz, R. Saigal and L. Vandenberghe},
{\em Handbook of Semidefinite Programming},
Kluwer, 2000.



\bibitem{YeLim18}
{\sc K.~Ye and L.-H.~Lim},
{\em Tensor network ranks},
\url{arXiv:1801.02662}, preprint, 2018


































































































































\end{thebibliography}
 \end{document}